\documentclass[12pt]{amsart}
\usepackage{amssymb,latexsym,amsmath,amsthm,enumerate,amscd,mathrsfs,yfonts,eucal,hyperref}
\usepackage{fullpage}
\input xy
\xyoption{all}

\newcommand{\ZZ}{\mathbb Z}

\newcommand{\NN}{\mathbb N}

\newcommand{\LL}{\mathbb L}
\newcommand{\cpt}{\mathbb K}
\newcommand{\tC}{\mathrm{C}}

\newcommand{\cA}{\mathcal{A}}

\newcommand{\cC}{\mathcal{C}}
\newcommand{\cD}{\mathcal{D}}

\newcommand{\cF}{\mathcal{F}}
\newcommand{\cH}{\mathcal{H}}

\newcommand{\cL}{\mathcal{L}}
\newcommand{\cM}{\mathcal{M}}
\newcommand{\cN}{\mathcal{N}}

\newcommand{\cR}{\mathcal{R}}
\newcommand{\cS}{\mathcal{S}}
\newcommand{\cT}{\mathcal{T}}

\newcommand{\cW}{\mathcal{W}}

\def\map{\rightarrow}
\DeclareMathOperator{\precolim}{colim}
\def\colim{\mathop{\precolim}}

\newcommand{\beq}{\begin{eqnarray}}
\newcommand{\beqn}{\begin{eqnarray*}}
\newcommand{\eeq}{\end{eqnarray}}
\newcommand{\eeqn}{\end{eqnarray*}}

\def\H{\mathrm{H}}

\def\ker{\mathrm{ker}}
\def\cone{\mathrm{cone}}
\def\Sing{\mathrm{Sing}}
\def\fib{\mathrm{fib}}
\def\hfib{\mathrm{hfib}}
\def\cofib{\mathrm{cofib}}

\def\Map{\mathrm{Map}}
\def\Hom{\mathrm{Hom}}

\def\Top{\mathrm{Top}}
\def\Ab{\mathtt{Ab}}
\def\all{\mathrm{all}}
\def\fin{\mathrm{fin}}

\def\Bon{\mathrm{Bon}}

\def\Ne{\mathtt{N}}

\def\Fun{\mathtt{Fun}}
\def\Catfinc{\mathtt{Cat_*^{fincolim}}}
\def\Funfinc{\mathtt{Fun_*^{fincolim}}}
\def\Catfinl{\mathtt{Cat_*^{finlim}}}
\def\Funfinl{\mathtt{Fun_*^{finlim}}}
\def\Catex{\mathtt{Cat^{Ex}}}
\def\Funex{\mathtt{Fun^{Ex}}}
\def\Prl{\mathtt{Pr^L_*}}
\def\Prr{\mathtt{Pr^R_*}}
\def\Prlex{\mathtt{Pr^L_{Ex}}}
\def\Prrex{\mathtt{Pr^R_{Ex}}}
\def\Funl{\mathtt{Fun^L}}
\def\Funr{\mathtt{Fun^R}}
\def\Catinfty{\mathtt{Cat^{\infty}}}
\def\LPK{\mathbb{L}\pi_K}

\def\Cgen{\mathtt{C^*}}

\def\id{\mathrm{id}}

\def\Csep{\mathtt{SC^*}}
\def\Csepinf{\mathtt{SC^*_\infty}}

\def\KK{\mathrm{KK}}
\def\E{\mathrm{E}}
\def\K{\mathrm{K}}
\def\op{\mathrm{op}}
\def\Ind{\mathrm{Ind}}

\def\Pro{\mathrm{Pro}}

\def\iNS{\mathtt{N}\mathcal{S_*}}

\def\Set{\mathrm{Set}}

\def\TPro{\mathrm{TPro}}

\def\SpN{\mathrm{Sp}^\mathbb{N}}
\def\Sp{\mathrm{Sp}}

\def\Ho{\mathrm{Ho}}
\def\BKK{\mathtt{BKK}}
\def\pro{\mathrm{pro}}
\def\ind{\mathrm{ind}}
\def\KKP{\mathtt{KK^{pro}}}
\def\KKI{\mathtt{KK^{ind}}}
\def\NSH{\mathtt{NSH}}

\def\iNS{\mathtt{N}{\mathcal{S_*}}}
\def\SW{\mathrm{SW}}
\def\sC{\text{$\sigma$-$C^*$}}
\def\skk{\text{$\sigma$-$\mathrm{kk}$}}

\theoremstyle{definition}
\newtheorem{thm}{Theorem}[section]

\newtheorem{lem}[thm]{Lemma}
\newtheorem{prop}[thm]{Proposition}
\newtheorem{cor}[thm]{Corollary}

\newtheorem{defn}[thm]{Definition}
\newtheorem{rem}[thm]{Remark}

\begin{document}

\title{Model structure on projective systems of $C^*$-algebras and bivariant homology theories}

\author[Barnea]{Ilan Barnea}
\address{Mathematical Institute, University of Muenster, Einsteinstrasse 62, 48149 Muenster, Germany.} \email{ilanbarnea770@gmail.com}

\author[Joachim]{Michael Joachim}
\address{Mathematical Institute, University of Muenster, Einsteinstrasse 62, 48149 Muenster, Germany.}
\email{joachim@math.uni-muenster.de}

\author[Mahanta]{Snigdhayan Mahanta}
\address{Fakult{\"a}t f{\"u}r Mathematik, Universit{\"a}t Regensburg, 93040 Regensburg, Germany.}
\email{snigdhayan.mahanta@mathematik.uni-regensburg.de}

\thanks{The first author was supported by the Alexander von Humboldt Foundation (Humboldt Professorship of Michael Weiss).}
\thanks{The third author was supported by the Deutsche Forschungsgemeinschaft (SFB 878 and SFB 1085), ERC through AdG 267079, and the Alexander von Humboldt Foundation (Humboldt Professorship of Michael Weiss).}

\begin{abstract}
Using the machinery of weak fibration categories due to Schlank and the first author, we construct a convenient model structure on the pro-category of separable $C^*$-algebras $\mathrm{Pro}(\mathtt{SC^*})$. The opposite of this model category models the $\infty$-category of pointed noncommutative spaces $\mathtt{N}\mathcal{S_*}$ defined by the third author. Our model structure on $\mathrm{Pro}(\mathtt{SC^*})$ extends the well-known category of fibrant objects structure on $\mathtt{SC^*}$. We show that the pro-category $\mathrm{Pro}(\mathtt{SC^*})$ also contains, as a full coreflective subcategory, the category of pro-$C^*$-algebras that are cofiltered limits of separable $C^*$-algebras. By stabilizing our model category we produce a general model categorical formalism for triangulated and bivariant homology theories of $C^*$-algebras (or, more generally, that of pointed noncommutative spaces), whose stable $\infty$-categorical counterparts were constructed earlier by the third author. Finally, we use our model structure to develop a bivariant $\mathrm{K}$-theory for all projective systems of separable $C^*$-algebras generalizing the construction of Bonkat and show that our theory naturally agrees with that of Bonkat under some reasonable assumptions.
\end{abstract}

\maketitle

\tableofcontents

\section{Introduction}

The Gel'fand--Na{\u{\i}}mark correspondence implies that the category of pointed compact Hausdorff (metrizable) spaces with pointed continuous maps is equivalent to the opposite category of commutative (separable) $C^*$-algebras with $*$-homomorphisms. In the realm of noncommutative geometry {\`a} la Connes, the category of \emph{all} (or \emph{separable}) $C^*$-algebras constitutes the basic setup. Let $\Cgen$ (resp. $\Csep$) denote the category of all (resp. separable) $C^*$-algebras with $*$-homomorphisms. It is natural to regard its opposite category as the category of \emph{noncommutative} pointed compact and Hausdorff (resp. metrizable) spaces. This category has been studied using tools from algebraic topology for a very long time (see for instance, \cite{KasEll}, \cite{RosNCT}, \cite{SchTop3}, \cite{CunKK}, \cite{ConBook}, \cite{Uuye}). In particular, there is a natural notion of homotopy that enables us to define homotopy equivalences in this context.

Quillen introduced model categories in \cite{Quillen} that provide a very general context in which it is possible to set up the basic machinery of homotopy theory. Thus an important question that arises is whether there exists a natural model structure on $\Cgen$ or $\Csep$; preferably one that also models the homotopy theory induced by the homotopy equivalences. This question was explicitly raised in Hovey's book  \cite[Problem 8.4]{HovBook}. It is not possible to build such a model structure directly on $\Cgen$ or $\Csep$ because of the following argument of Andersen-Grodal \cite[Corollary 4.7]{AndGro} (see also \cite{Phi2,Phi}):
We know how to construct the analogue of the suspension functor
$$\Sigma:\Ho(\Cgen^{\op})\map\Ho(\Cgen^{\op}).$$
Namely, if $A\in\Cgen^\op$, then
$$\Sigma A:=S^1\wedge A$$
is just the separable $C^*$-algebra of \emph{pointed} continuous maps from $S^1$ to $(A,0)$. However, this functor does not have a right adjoint $\Omega$. An intuitive reason for this is the inherent compactness of the objects in $\Cgen$, since the functor $\Omega$ can take a compact space to a non-compact one (for example, $\Omega S^1\simeq \mathbb{Z}$). We thus need to extend the category $\Cgen$ to include non-compact noncommutative spaces in order to put a model structure on it. For practical applications it is often sufficient to restrict one's attention to separable $C^*$-algebras. Thus keeping in mind the contravariant nature of the Gel'fand--Na{\u{\i}}mark duality, we formulate the following problem: Find a category $\cD$, that contains $\Csep$ as a full subcategory, and construct a model structure on $\cD$ such that

\begin{enumerate}[(i)]
 \item The inclusion functor $\Csep\map\cD$ sends homotopy equivalences to weak equivalences and the resulting map $\Ho\Csep\map\Ho\cD$ is fully faithful.

 \item The category $\cD$ is \emph{as close to $\Csep$ as possible}, preferably already known and studied in $C^*$-algebra theory.

 \item The model category $\cD$ is simplicial, proper and cocombinatorial (i.e., the opposite model category is combinatorial).
\end{enumerate}
The first item codifies the requirement that one must build a homotopy theory for the prevalent notion in the literature. The second item takes into account the requirement that there should be minimal deviation from the well established theory of $C^*$-algebras. The third item stipulates that the model category possess features that facilitate homotopy theoretic constructions therein (see Appendix \ref{a:model} for more detail).

We are aware of the following different model categories in the context of $C^*$-algebras (the authors apologise for any omission due to ignorance):
\begin{enumerate}
 \item The homotopy theory of cubical $C^*$-spaces by {\O}stv{\ae}r \cite{Ost},

 \item The model structure on $\nu$-sequentially complete l.m.c-$C^*$-algebras by Joachim--Johnson \cite{JoaJoh},

 \item The model category (or the $\infty$-category) of pointed noncommutative spaces by the third author \cite{MahNSH},

 \item The Morita homotopy theory of $C^*$-categories by Dell'Ambrogio--Tabuada \cite{DelTab}, and

 \item The operadic model structure on the topos $\mathcal{P}(\mathtt{SC^*_{un}}^\op)$ by the third author, where $\mathtt{SC^*_{un}}$ denotes the category of nonzero separable and unital $C^*$-algebras with unital $*$-homomorphisms \cite{MahDrawing}.
\end{enumerate}

In item (1), {\O}stv{\ae}r constructs a model structure on cubical set valued presheaves on $\Csep$. He begins with the projective model structure, and then the appropriate model category is obtained by successive Bousfield localizations. The end result has a flavour of the motivic (unstable) model category. The underlying category of this model category is not so well known in the theory of $C^*$-algebras and is much bigger than the candidate that we put forward. Thus this model category does not satisfy the second criterion above.

The approach in item (2) relies on the quasi-homomorphism picture for $\KK$-theory due to Cuntz \cite{CunKK} to build a model category, whose homotopy category contains Kasparov $\KK$-category fully faithfully. The enlargement of the category of $C^*$-algebras is carefully chosen by the authors in \cite{JoaJoh} so that it permits the small object argument leading to the construction of a cofibrantly generated model structure. Evidently it does not satisfy the first criterion mentioned above. The same comment applies to the approaches in items (4) and (5). Actually the model category of item (5) acts as a bridge between dendroidal sets and noncommutative spaces; more precisely, it acts as a bridge only at the level of underlying $\infty$-categories of $\infty$-operads and noncommutative spaces.

Before turning our attention to item (3) let us mention that the most straightforward way that extends the category of $C^*$-algebras to include non-compact noncommutative spaces is to consider the classical notion of pro-$C^*$-algebras (\cite{Phi,Phi2}).
These are topological $*$-algebras that are cofiltered limits of $C^*$-algebras (in the category of topological $*$-algebras). Commutative unital pro-$C^*$-algebras roughly correspond to completely Hausdorff compactly generated spaces (strictly speaking, one should consider completely Hausdorff \textbf{quasitopological} spaces). These objects are very close to $C^*$-algebras (in particular, they are also topological $*$-algebras), and were studied in $C^*$-algebra theory (so they certainly satisfy the second criterion above). There is also a notion of homotopy equivalence between pro-$C^*$-algebras, so that homotopy equivalences are natural candidates for weak equivalences. Hence if we could define a model structure with these weak equivalences, the first criterion above would also be satisfied. This attempted model structure would be similar to the Str{\o}m model structure on topological spaces \cite{Strom}, where the weak equivalences are the homotopy equivalences, and it is quite conceivable that Str{\o}m's construction would generalize to the category of pro-$C^*$-algebras. However, it is very likely that the resulting model structure would fail to be cocombinatorial, and thus violate criterion three above, much like the Str{\o}m model structure  \cite[Remark 4.7]{Raptis}.

Let us now explain how our article complements the approach of item number (3). Being an active area of research, the theory of $\infty$-categories is considered nowadays to be the most appropriate and conceptual environment for using homotopy theoretic tools in a generalized context. There are two natural ways of considering $\Csep$ as an $\infty$-category:
\begin{enumerate}
\item We can consider $\Csep$ as a topologically enriched category, where for every $A,B\in\Csep$ we endow the set of $*$-homomorphisms $\Csep(A,B)$ with the \emph{point norm topology}. Then we can take topological nerve of this topological category as in \cite[Section 1.1.5]{Lur}. This approach was taken by the third author in \cite{MahNSH}.
\item We can consider $\Csep$ as a relative category, with the weak equivalences given by the homotopy equivalences. Then we can take $\infty$-localization of this relative category. (We refer the reader to Appendix \ref{ss:infinity} for  the definition of a relative category and the $\infty$-localization of a relative category.) This relative category was considered in \cite{AndGro,SchTop3,Uuye}.
\end{enumerate}
We are going to show in Proposition \ref{p:equiv_infinity} that these two ways are equivalent. Let us denote by $\Csepinf$ the $\infty$-category obtained by either of the two equivalent ways above. The $\infty$-category $\Csepinf$ is not a very convenient one from an $\infty$-categorical perspective, since it does not permit many natural constructions. For example, while $\Csepinf$ admits finite $\infty$-limits  \cite[Proposition 2.7]{MahNSH}, it does not have finite $\infty$-colimits. Indeed, if it did possess finite $\infty$-colimits, one could define an adjoint pair of $\infty$-categories, as in Appendix \ref{s:stab_infinity}:
$$\Sigma_{\Csepinf}:\Csepinf\rightleftarrows \Csepinf:\Omega_{\Csepinf}.$$
Being an adjoint pair of $\infty$-categories, it would descend to an adjoint pair on their homotopy categories, thereby contradicting the argument of Andersen-Grodal mentioned above. This is why it is desirable to embed $\Csepinf$ in a bigger $\infty$-category, which is complete and cocomplete in an $\infty$-categorical sense. One of the most convenient types of $\infty$-categories is that of \emph{presentable $\infty$-categories}, or even more particularly, \emph{compactly generated $\infty$-categories} (see \cite[Chapter 5]{Lur}). There is a very natural procedure to embed $\Csepinf^\op$ in a compactly generated $\infty$-category, viz., since $\Csepinf^\op$ admits finite colimits, one may simply take its $\infty$-categorical ind-completion $\Ind(\Csepinf^\op)$. This is a very elegant solution, since $\Ind(\Csepinf^\op)$ is generated by the objects in $\Csepinf^\op$, which become compact inside it. The $\infty$-category $\Ind(\Csepinf^\op)$ was called the $\infty$-category of pointed noncommutative spaces (that are not necessarily compact) by the third author and denoted by $\iNS$ in \cite{MahNSH}.

While the theory of $\infty$-categories is very conceptual and enables us to prove theorems using universal properties, when it comes to concrete calculations it is in many cases quite abstract. For this purpose it is beneficial to have a convenient model structure that models the $\infty$-category of pointed noncommutative spaces (see Appendix \ref{ss:infinity} for the exact meaning of this).
This leads us to the model structure constructed in this paper, which fulfils this requirement and also seems to satisfy all the above-mentioned criteria (cf. Remark \ref{ModelComparison} below).

In this paper we construct a model structure on the category of \emph{projective systems} of separable $C^*$-algebras, which we denote $\Pro(\Csep)$. This is done in Theorem \ref{t:proC_main}, where it is also shown that this model structure satisfies the third criterion above. The construction of our model structure is based on a general method for constructing model structures on pro-categories that was developed by T. Schlank and the first author in \cite{BarSch1,BarSch2}. In these papers the concept of a \emph{weak fibration category} was introduced. A weak fibration structure on a category is much weaker than a model structure. It is shown in \cite{BarSch1,BarSch2} that a small weak fibration structure on a category naturally induces a model structure on its pro-category, provided the induced weak equivalences satisfy the two out of three property. The verification of this two out of three property is usually not an easy task. In order to show this in our case we apply a result proved by the first author in \cite{Bar} that gives sufficient intrinsic conditions on a weak fibration category for this two out of three property to hold.

A weak fibration category is a triple $(\cC,\cW,\cF)$ consisting of a category $\cC$ and two subcategories $\cW$ and $\cF$, called weak equivalences and fibrations, satisfying certain axioms (weaker than those of a model category). This notion is closely related to Brown's notion of a \emph{category of fibrant objects} \cite{Bro} and Baues's notion of a fibration category \cite{Bau}, which were introduced as more flexible structures than a model category that permit abstract homotopy theory. Andersen--Grodal defined a structure of a Baues fibration category on $\Csep$ in their unpublished paper \cite{AndGro}. Uuye later gave a different proof \cite{Uuye}. Both are building upon the earlier work of Schochet \cite{SchTop3}. We explain why this construction also constitutes a weak fibration structure on $\Csep$, and use it as a starting point for constructing our model structure on $\Pro(\Csep)$, as explained in the previous paragraph. Thus we extend the fibration structure on $\Csep$ to a much more powerful model structure on $\Pro(\Csep)$.

The category of projective systems (as well as inductive systems and some other diagram categories) of $C^*$-algebras has already been studied in the literature on $C^*$-algebras. For instance, it was considered by Bonkat \cite{Bon} with applications to bivariant $\K$-theory and by Puschnigg \cite{PusHL} and Meyer \cite{MeyCycHom} with applications to bivariant cyclic homology theories. In combination with Kirchberg's techniques, diagrams of (separable) $C^*$-algebras have since then been used effectively in various classification problems (see for instance, \cite{EilRes,EilResRui,MeyNesFK}).

Moreover, we are also able to build a bridge between the objects of our model category and the classical notion of pro-$C^*$-algebras mentioned above. In particular, we show in Proposition \ref{p:pro adjoint} that the underlying category of our model category contains, as a full coreflective subcategory, a very large category of pro-$C^*$-algebras (namely, those that are cofiltered limits of \textbf{separable} $C^*$-algebras). Thus we can say that our model structure satisfies the second criterion above.

We further show in Proposition \ref{p:noncomm_space} that the opposite category of our model category models the $\infty$-category of pointed noncommutative spaces $\Ne\cS_*$ described above. We do this using a general result proved by the first author, Y. Harpaz and G. Horel in \cite{BarHarHor}, which connects the model structure on a pro-category defined in \cite{BarSch1,BarSch2} with the $\infty$-categorical pro-construction. A direct consequence of this is that our model structure also satisfies the first criterion above.

\begin{rem} \label{ModelComparison}
Since the $\infty$-category $\Ne\cS_*$ is presentable, there is a general construction giving a combinatorial simplicial model category $\cM$, that models $\Ne\cS_*$ (see \cite[Proposition A.3.7.6]{Lur}). However, the underlying category of the model category produced by this general construction is much bigger than the one that we construct here. Indeed, the underlying category of $\cM$ is the category of simplicial presheaves on a small simplicial category containing $\Csep^{\op}$; whereas the (opposite of the) underlying category of our model structure can be realised as a \textbf{full subcategory} of the category of \textbf{usual} (set valued) presheaves on $\Csep^{\op}$. From the viewpoint of applications to classification problems and computations of various bivariant homology theories the simplicity and convenience of our new model structure is quite significant.
\end{rem}

After constructing the $\infty$-category of pointed noncommutative spaces $\Ne\cS_*$ as a starting point, the third author constructed in \cite{MahNSH} several bivariant homology theories on $\Ne\cS_*$ using $\infty$-categorical tools such as \emph{stabilization} and \emph{localization}. These theories extend the applicability of some known theories on the category of separable $C^*$-algebras, like $\K$-theory and noncommutative stable homotopy theory. Using our model structure, that models pointed noncommutative spaces by projective systems of separable $C^*$-algebras, all these homology theories become homology theories for projective systems of separable $C^*$-algebras. Our constructions also develop a parallel world of stable model categories that model the stable $\infty$-categories constructed by the third author in \cite{MahNSH,MahNSp,MahColoc} (see Proposition \ref{p:homology_main_model}).

In particular, using the general construction mentioned above, we obtain a bivariant $\K$-theory for projective systems of separable $C^*$-algebras. In \cite{Bon}, Bonkat also constructed a bivariant $\K$-theory for certain types of projective systems of separable $C^*$-algebras, using analytic tools extending the Kasparov bimodule picture. In Theorem \ref{t:main Bonkat}, we use our model structure to show that in certain cases, our bivariant $\K$-theory agrees with Bonkat's construction. However, note that our bivariant $\K$-theory applies to \emph{all} projective systems of separable $C^*$-algebras, while Bonkat's construction only applies to projective systems that have surjective connecting homomorphisms and admit a countable cofinal subsystem. Furthermore we show that our $\K$-theory has better formal properties (see Theorems \ref{t:LPK satisfy}, \ref{t:Our lim1} and \ref{t:Our lim}). Further applications of the framework developed in this article will appear elsewhere.

{\bf Overview of the paper.} In Section \ref{s:prelim} we review some of the necessary background on pro-categories and homotopy theory in pro-categories. In particular, we recall the definition of the pro-category of a general category, as well as some related theory. In the homotopical part, we recall the definition of a simplicial weak fibration category and state Theorem \ref{t:tool_main}, which is the main tool for constructing our model structure. We end the section by considering the relation between the model structure on a pro-category defined in Theorem \ref{t:tool_main} with the $\infty$-categorical pro-construction.

In Section \ref{s:model C} we construct our model structure on the category of projective systems of separable $C^*$-algebras. We begin in Subsection \ref{ss:SC WFC} by defining a simplicial weak fibration structure on the category $\Csep$ of separable $C^*$-algebras (see Propositions \ref{p:SC_WFC} and \ref{p:simplicial}). In Subsection \ref{ss:model pro C} we use the results of the previous subsection and Theorem \ref{t:tool_main} to construct our model structure on $\Pro(\Csep)$ (see Theorem \ref{t:proC_main}). In Subsection \ref{ss:relation infty}, we begin by showing that the two ways to look at $\Csep$ as an $\infty$-category mentioned above are equivalent (see Proposition \ref{p:equiv_infinity}). We then deduce in Propositions \ref{p:noncomm_space} that the underlying $\infty$-category of our model structure is naturally equivalent to the opposite $\infty$-category of pointed noncommutative spaces defined in \cite{MahNSH}. We also obtain a stable version of this last result in Proposition \ref{p:N_spectra}. We end this section with Subsection \ref{s:connect}, in which we connect the underlying category of our model structure on $\Pro(\Csep)$ with the more classical category of pro-$C^*$-algebras (\cite{Phi},\cite{Phi2}). More precisely, we show in Proposition \ref{p:pro adjoint} that $\Pro(\Csep)$ contains, as a full coreflective subcategory, a very large category of pro-$C^*$-algebras (namely, those that are cofiltered limits of \textbf{separable} $C^*$-algebras).

In Section \ref{s:homology} we consider bivariant homology theories on projective systems of separable $C^*$-algebras. We begin with Subsection \ref{ss:define homology}, in which we define the notion of a triangulated homology theory on a pointed cocomplete $\infty$-category. In Subsection \ref{ss:THT} we
recall a construction defined by the third author in \cite{MahNSH}, which associates a triangulated homology theory on the $\infty$-category of pointed noncommutative spaces to any set of morphisms in $\Csep$ (see Theorem \ref{p:homology_main}). For any set of morphisms in $\Csep$, by taking the opposite category and using our model structure, we get a bivariant homology theory which is applicable to all projective systems of separable $C^*$-algebras. We then transform this construction, which uses the language of $\infty$-categories, to the world of model categories (see Theorem \ref{p:homology_main_model}). We end in Subsection \ref{ss:homology examples} by considering several examples of this general construction. In particular, we construct a bivariant $\K$-theory category for projective systems of separable $C^*$-algebras, and show that it extends Kasparov's bivariant $\K$-theory. We also show how to use our model structure in order to obtain a representing projective system for $\K$-theory.

Originally Bonkat constructed a bivariant $\K$-theory for certain types of projective systems of separable $C^*$-algebras \cite{Bon}. In Section \ref{s:Bonkat} we compare the bivariant $\K$-theory for projective systems constructed in \ref{ss:homology examples} with Bonkat's construction.
We begin with Theorem \ref{t:LPK satisfy}, in which we show that our bivariant $\K$-theory satisfies the same defining properties as Bonkat's, namely, homotopy invariance, $C^*$-stability and split exactness. While our bivariant $\K$-theory satisfies these properties for \emph{all} projective systems of separable $C^*$-algebras, Bonkat's construction only applies to projective systems that have surjective connecting homomorphisms and admit a countable cofinal subsystem.
We then show, in Theorems \ref{t:Our lim1} and \ref{t:Our lim}, that the main calculational tools of Bonkat's $\K$-theory also hold for ours, and in fact, under less restrictive assumptions. We end with Theorem \ref{t:main Bonkat} in which we use the results above to show that in certain cases, our bivariant $\K$-theory agrees with Bonkat's construction. In this last section we use our model structure and its properties in an essential way.

\medskip
\noindent
{\bf Notations and conventions:} Throughout the article we use the language of model categories and that of $\infty$-categories. We refer the readers to \cite{HovBook} or \cite{Hir} for the prerequisites from the theory of model categories. For the benefit of the readers we have gathered some of the main results that we need about model categories in Appendix \ref{a:model}. By an $\infty$-category we mean the quasicategory model of Joyal and Lurie (see \cite{Joyal,Lur}). We have also compiled some of the main results that we need about $\infty$-categories in Appendix \ref{a:infinity}.

We denote by $\Csep$ the category of separable $C^*$-algebras and $*$-homomorphisms between them. Whenever we mention a tensor product on $\Csep$ we mean the maximal $C^*$-tensor product.
Whenever we mention a morphism between objects in $\Csep$ we mean a $*$-homomorphism.

\medskip
\noindent
{\bf Acknowledgements:}
The first author would like to thank Tomer M. Schlank for useful conversations. The third author has benefited from the hospitality of Max Planck Institute for Mathematics, Bonn and Hausdorff Research Institute for Mathematics, Bonn under various stages of development of this project.

\section{Preliminaries: homotopy theory in pro-categories}\label{s:prelim}
In this section we review some of the necessary background on pro-categories and homotopy theory in pro-categories. Standard references on pro-categories include \cite{SGA4-I} and \cite{AM}. For the homotopical parts the reader is referred to \cite{BarSch0,BarSch1,BarSch2,Bar}.  See also \cite{EH} and \cite{Isa}.

\subsection{Pro-categories}
In this subsection we bring general background on pro-categories.

\begin{defn}\label{d:cofiltered}
A category $I$ is called \emph{cofiltered} if the following conditions are satisfied:
\begin{enumerate}
\item $I$ is non-empty.
\item for every pair of objects $s,t \in I$, there exists an object $u\in I$, together with
morphisms $u\to s$ and $u\to t$.
\item for every pair of morphisms $f,g:s\to t$ in $I$, there exists a morphism $h:u\to s$ in $I$ such that $f\circ h=g\circ h$.
\end{enumerate}
\end{defn}

If $T$ is a poset, then we view $T$ as a category which has a single morphism $u\to v$ iff $u\geq v$. Thus, a poset $T$ is cofiltered iff $T$ is non-empty, and for every $a,b\in T$ there exists $c\in T$ such that $c\geq a, b$. A cofiltered poset will also be called \emph{directed}. Additionally, in the following, instead of saying \emph{a directed poset} we will just say \emph{a directed set}.

\begin{defn}\label{def CDS}
A poset $T$ is called \emph{cofinite} if for every element $x$ in $T$ the set $T_x:=\{z\in T\;|\; z \leq x\}$ is finite.
\end{defn}

A category is called \emph{small} if it has only a set of objects and a set of morphisms.

\begin{defn}
Let ${\cC}$ be a category. The category $\Pro({\cC})$ has as objects all diagrams in $\cC$ of the form $I\to \cC$ such that $I$ is small and cofiltered (see Definition ~\ref{d:cofiltered}). The morphisms are defined by the formula
$$\Hom_{\Pro({\cC})}(X,Y):=\lim_s \colim_t \Hom_{{\cC}}(X_t,Y_s).$$
Composition of morphisms is defined in the obvious way.
\end{defn}

Thus, if $X:I\to {\cC}$ and $Y:J\to {\cC}$ are objects in $\Pro({\cC})$, providing a morphism $X\to Y$ means specifying for every $s$ in $J$ an object $t$ in $I$ and a morphism $X_t\to Y_s$ in ${\cC}$. These morphisms should satisfy some compatibility condition. In particular, if $p:J\to I$ is a functor, and $\phi:p^*X:=X\circ p\to Y$ is a natural transformation, then the pair $(p,\phi)$ determines a morphism $\nu_{p,\phi}:X\to Y$ in $\Pro(\cC)$ (for every $s$ in $J$ we take the morphism $\phi_s:X_{p(s)}\to Y_s$). Taking $Y=p^*X$ and $\phi$ to be the identity natural transformation, we see that $p$ determines a morphism $\nu_{p,X}:X\to p^*X$ in $\Pro(\cC)$. If $I=J$ and we take $p=\id$, we see that every natural transformation $X\to Y$ determines a morphism in $\Pro(\cC)$.

The word pro-object refers to objects of pro-categories. A \emph{simple} pro-object
is one indexed by the category with one object and one (identity) map. Note that for any category ${\cC}$, $\Pro({\cC})$ contains ${\cC}$ as the full subcategory spanned by the simple objects. We will thus abuse notation and treat $\cC$ as a full subcategory of $\Pro({\cC})$.

\begin{defn}\label{d:cofinal}
Let $p:J\to I$ be a functor between small categories. The functor $p$ is said to be \emph{(left) cofinal} if for every $i$ in $I$ the over category ${p}_{/i}$ is nonempty and connected (This means that the geometric realization is a nonempty connected space).
\end{defn}

Cofinal functors play an important role in the theory of pro-categories mainly because of the following well known lemma:

\begin{lem}\label{l:cofinal}
Let $p:J\to I$ be a cofinal functor between small cofiltered categories, and let $X:I\to \cC$ be an object in $\Pro(\cC)$. Then $\nu_{p,X}:X\to p^*X$ is an isomorphism in $\Pro(\cC)$.
\end{lem}

The following lemma can be found in \cite[Proposition 8.1.6]{SGA4-I} and in \cite[Corollary 3.11]{BarSch0}:
\begin{lem}\label{l:cofinal_CDS}
Let $I$ be a small cofiltered category. Then there exists a small
cofinite directed set $A$ and a cofinal functor $A\to I$.
\end{lem}

\begin{defn}\label{def natural}
Let ${\cC}$ be a category with finite limits, $M$ a class of morphisms in ${\cC}$, $I$ a small category, and $F:X\to Y$ a morphism in ${\cC}^I$. Then:
\begin{enumerate}
\item The map  $F$ will be called a \emph{levelwise} $M$-\emph{map}, if for every $i$ in $I$ the morphism $X_i\to Y_i$ is in $M$. We will denote this by $F\in Lw(M)$.
\item The map  $F$ will be called a \emph{special} $M$-\emph{map}, if the following hold:
    \begin{enumerate}
    \item The indexing category $I$ is a cofinite poset (see Definition \ref{def CDS}).
    \item The natural map $X_t \to Y_t \times_{\lim_{s<t} Y_s} \lim_{s<t} X_s $ is in $M$, for every $t$ in $ I$.
    \end{enumerate}
    We will denote this by $F\in Sp(M)$.

\end{enumerate}
\end{defn}

\begin{defn}\label{d:lift}
Let $\cC$ be a category and let $f:A\to B$ and $g:X\to Y$ be morphisms in $\cC$. Then we say that $f$ has the \emph{left lifting property} with respect to $g$, or equivalently, that $g$ has the \emph{right lifting property} with respect to $f$, if in every commutative square of the form
$$\xymatrix{A\ar[r]\ar[d]_f & X\ar[d]^g\\
            B\ar[r] & Y}$$
we have a lift $B\to X$, making the diagram commutative.
\end{defn}

\begin{defn}\label{def mor}
Let ${\cC}$ be a category and let $M$ be a class of morphisms in ${\cC}$.
\begin{enumerate}
\item We denote by $R(M)$ the class of morphisms in ${\cC}$ that are retracts of morphisms in $M$. Note that $R(R(M))=R(M)$.
\item We denote by $M^{\perp}$ (resp. ${}^{\perp}M$) the class of morphisms in ${\cC}$ having the right (resp. left) lifting property with respect to all the morphisms in $M$.
\item We denote by $Lw^{\cong}(M)$ the class of morphisms in $\Pro({\cC})$ that are \textbf{isomorphic} to a morphism that comes from a natural transformation which is a levelwise $M$-map.
\item If $\cC$ has finite limits, we denote by $Sp^{\cong}(M)$ the class of morphisms in $\Pro({\cC})$ that are \textbf{isomorphic} to a morphism that comes from a natural transformation which is a special $M$-map.
\end{enumerate}
\end{defn}

Everything we did so far (and throughout this paper) is completely dualizable. Thus we can define:
\begin{defn}\label{d:filtered}
A category $I$ is called \emph{filtered} if the following conditions are satisfied:
\begin{enumerate}
\item $I$ is non-empty.
\item for every pair of objects $s,t \in I$, there exists an object $u\in I$, together with
morphisms $s\to u$ and $t\to u$.
\item for every pair of morphisms $f,g:s\to t$ in $I$, there exists a morphism $h:t\to u$ in $I$ such that $h\circ f=h\circ g$.
\end{enumerate}
\end{defn}

The dual to the notion of a pro-category is the notion of an ind-category:

\begin{defn}
Let ${\cC}$ be a category. The category $\Ind({\cC})$ has as objects all diagrams in $\cC$ of the form $I\to \cC$ such that $I$ is small and filtered (see Definition ~\ref{d:filtered}). The morphisms are defined by the formula:
$$\Hom_{\Ind({\cC})}(X,Y):=\lim_s \colim_t \Hom_{{\cC}}(X_s,Y_t).$$
Composition of morphisms is defined in the obvious way.
\end{defn}

Clearly for every category $\cC$ we have a natural isomorphism of categories $\Ind(\cC)^{\op}\cong \Pro(\cC^{\op})$.

In general, we are not going to write the dual to every definition or theorem explicitly, only in certain cases.

\subsection{From a weak fibration category to a model category}
In this subsection we discuss the construction of model structures on pro-categories.

\begin{defn}
Let ${\cC}$ be category with finite limits, and let ${\cM}\subseteq{\cC}$ be a subcategory. We say
that ${\cM}$ is \emph{closed under base change} if whenever we have a pullback square:
\[
\xymatrix{A\ar[d]_g\ar[r] & B\ar[d]^f\\
C\ar[r] & D}
\]
such that $f$ is in ${\cM}$, then $g$ is in ${\cM}$.
\end{defn}

\begin{defn}\label{d:weak_fib}
A \emph{weak fibration category} is a category ${\cC}$ with an additional
structure of two subcategories:
$${\cF}, {\cW} \subseteq {\cC}$$
that contain all the isomorphisms such that the following conditions are satisfied:
\begin{enumerate}
\item ${\cC}$ has all finite limits.
\item ${\cW}$ has the 2 out of 3 property.
\item The subcategories ${\cF}$ and ${\cF}\cap {\cW}$ are closed under base change.
\item Every map $A\to B $ in ${\cC}$ can be factored as $A\xrightarrow{f} C\xrightarrow{g} B $,
where $f$ is in ${\cW}$ and $g$ is in ${\cF}$.
\end{enumerate}
The maps in ${\cF}$ are called \emph{fibrations}, the maps in ${\cW}$ are called \emph{weak equivalences}, and the maps in $\cF\cap\cW$ are called \emph{acyclic fibrations}.
\end{defn}

Let $\cS_{\fin}$ denote the category of finite simplicial sets, that is, simplicial sets having a finite number of non-degenerate simplicies.
Note that there is a natural equivalence of categories $\Ind(\cS_{\fin})\xrightarrow{\sim}\cS$, given by taking colimits (see \cite{AR}).
We define a map in $\cS_{\fin}$ to be a cofibration or a weak equivalence, if it is so in the usual model structure on simplicial sets.

\begin{defn}\label{d:simplicial_weak}
A \emph{simplicial weak fibration category} is a weak fibration category $\cC$  together with a bifunctor $\hom(-,-):\cS_{\fin}^{\op}\times\cC\to \cC$ and coherent natural isomorphisms
$$\hom(L, \hom(K, X))\cong \hom(K\times L, X),$$
$$\hom(\Delta^0, X)\cong X,$$
for $X$ in $\cC$ and $K,L$ in $\cS_{\fin}$, such that:
\begin{enumerate}
\item The bifunctor $\hom$ commutes with finite limits in every variable separately.
\item For every cofibration $j:K\to L$ in $\cS_{\fin}$ and every fibration $p:A\to B$ in $\cC$, the induced map:
$$\hom(L,A)\to \hom(K,A) \prod_{\hom(K,B)}\hom(L,B)$$
is a fibration (in $\cC$), which is acyclic if either $j$ or $p$ is.
\end{enumerate}
\end{defn}

We now give our main tool for constructing our model structure. This is the main theorem in the paper \cite{Bar} by the first author. It is based on earlier joint work with Tomer M. Schlank \cite{BarSch0,BarSch1,BarSch2}. See also \cite{EH} and \cite{Isa} for related results. Note that the result in \cite{Bar} is stated for the dual ind-picture, but we bring it here in the form appropriate to the application that we need.

\begin{thm}[{\cite[Theorem 4.13]{Bar}}]\label{t:tool_main}
Let $({\cC},{\cW},{\cF})$ be a small simplicial weak fibration category  that satisfies the following conditions:
\begin{enumerate}
\item $\cC$ has finite colimits.
\item Every object in $\cC$ is fibrant.
\item A map in $\cC$ that is a homotopy equivalence in the simplicial category $\cC$ is also a weak equivalence.
\item Every acyclic fibration in $\cC$ admits a section.
\end{enumerate}
Then there exists a simplicial model category structure on $\Pro(\cC)$ such that:
\begin{enumerate}
\item The weak equivalences are $\mathbf{W} := Lw^{\cong}(\cW)$.
\item The fibrations are $\mathbf{F} := R(Sp^{\cong}(\cF))$.
\item The cofibrations are $\mathbf{C} :=  {}^{\perp}(\cF\cap \cW)$.
\end{enumerate}

Moreover, this model category is cocombinatorial, with set of generating fibrations $\cF$ and set of generating acyclic fibrations $\cF\cap \cW$.

The model category $\Pro(\cC)$ has the following further properties:
\begin{enumerate}
\item The acyclic fibrations are given by $\mathbf{F}\cap\mathbf{W} = R(Sp^{\cong}(\cF\cap\cW))$.
\item Every object in $\Pro(\cC)$ is cofibrant.
\item $\Pro(\cC)$ is proper.
\end{enumerate}
\end{thm}

\begin{rem}
The simplicial structure on $\Pro(\cC)$ in the theorem above is given by the natural prolongation of the cotensor action of $\cS_{\fin}$ on $\cC$, using the natural equivalence of categories $\cS\simeq\Ind(\cS_{\fin})$. Namely, if $K=\{K_i\}_{i\in I}$ is an object in $\cS\simeq\Ind(\cS_{\fin})$ and $A=\{A_j\}_{j\in J}$ is an object in $\Pro(\cC)$ then
$$\hom(K,A)=\{\hom(K_{i}, A_{j})\}_{(i,j)\in I^{\op}\times J}\in\Pro(\cC).$$
\end{rem}

\begin{rem}\label{r:tool_main}
If $(\cM,\cW,\cF,\cC)$ is any model category, then $(\cM^{\op},\cW^{\op},\cC^{\op},\cF^{\op})$ is also a model category. Thus, if $({\cC},{\cW},{\cF})$ is a weak fibration category satisfying the hypothesis of Theorem \ref{t:tool_main}, so that there is an induced simplicial model structure on $\Pro(\cC)$, there is also an induced simplicial model structure on $\Pro(\cC)^{\op}\cong\Ind(\cC^{\op})$, with properties dual to those stated in Theorem \ref{t:tool_main}.
\end{rem}

\subsection{Relation to pro-$\infty$-categories}
We finish this preliminary section by connecting the model structure of Theorem \ref{t:tool_main} with the $\infty$-categorical construction of the pro-category. In \cite[Section 5.3]{Lur}, Lurie defines the ind-category of a small $\infty$-category. The pro-category of a small $\infty$-category $\cC$ can be simply defined as $\Pro(\cC):=\Ind(\cC^{\op})^{\op}$.

Let $({\cC},{\cW},{\cF})$ be a weak fibration category satisfying the conditions of Theorem \ref{t:tool_main}. Clearly, we have a natural relative functor
$$({\cC},{\cW})\to (\Pro(\cC),Lw^{\cong}(\cW)).$$
This relative functor induces an $\infty$-functor between the $\infty$-localizations
$$\cC_\infty\to \Pro(\cC)_\infty.$$
(See Appendix \ref{ss:infinity} for  the definition of a relative category and the $\infty$-localization of a relative category.) The following theorem is a corollary of the main result in \cite{BarHarHor}:
\begin{thm}\label{t:BarHarHor}
Extending the natural functor $\cC_\infty\to \Pro(\cC)_\infty$ according to the universal property of the $\infty$-categorical pro-construction gives an equivalence of $\infty$ categories
$$\Pro(\cC_\infty)\simeq \Pro(\cC)_\infty.$$
In particular, the natural functor $\cC_\infty\to \Pro(\cC)_\infty$ is derived fully faithful.
\end{thm}

\begin{rem}\label{r:BarHarHor}
By Remark  \ref{r:tool_main} we have an induced model structure on $\Ind(\cC^{\op})$. It thus follows from Theorem \ref{t:BarHarHor} that there is a natural equivalence of $\infty$ categories
$$\Ind(\cC_\infty^{\op})\simeq \Ind(\cC^{\op})_\infty.$$
\end{rem}

\section{Model structure on the pro-category of $C^*$-algebras}\label{s:model C}

In this section we construct our model structure on the category $\Pro(\Csep)$, where $\Csep$ is the category of separable $C^*$-algebras.

\subsection{$\Csep$ as a weak fibration category}\label{ss:SC WFC}

\begin{defn}\label{d:SC}
Let $\Csep$ denote the category of separable $C^*$-algebras and $*$-homomorphisms between them.
\end{defn}

\begin{rem}\label{r:Top enrich}
As noted in \cite{Uuye}, the category $\Csep$ is naturally enriched over $\Top$, the Cartesian closed category of compactly generated weakly Hausdorff topological spaces. Indeed, if $A,B\in\Csep$, we can give $\Hom_{\Csep}(A,B)$ the subspace topology of the space of all continuous maps $\Top(A,B)$, endowed with the compact open topology. We denote this \emph{space} by $\Csep(A,B)$. Since $A$ is separable it follows that $\Csep(A,B)$ is metrizable and hence compactly generated Hausdorff (see \cite[Remark 2.1]{Uuye}).
\end{rem}

\begin{rem}\label{r:conventions}
The category $\Csep$ is essentially small, since any separable $C^*$-algebra is isomorphic to a sub $C^*$-algebra of the $C^*$-algebra of bounded operators on $\ell^2$. We can therefore assume that we are working with an equivalent small category, and we will do so without mentioning.
\end{rem}

We now define a structure of a simplicial weak fibration category on $\Csep$.
\begin{defn}[{\cite[Definition 2.14]{Uuye}}]
A map $p: A\map B$ in $\Csep$ is called a {\em Schochet fibration} if for every $D\in\Csep$ and every commutative diagram of the form
$$\xymatrix{
\{0\} \ar[d] \ar[r] & \Csep(D,A)\ar[d]^{p_*}\\
[0,1]\ar[r] & \Csep(D,B),}$$
there exists a lift $[0,1]\map \Csep(D,A)$.
\end{defn}

\begin{rem}
The definition of a Schochet fibration was originally introduced by Schochet in \cite{SchTop3}, where such maps were called {\em cofibrations}.
\end{rem}

\begin{defn}\label{d:action}
For every $K$ in $\cS_{\fin}$ and $A$ in $\Csep$ we denote by
$$\hom(K,A):=\tC(|K|)\otimes A=\tC(|K|,A)$$
the separable $C^*$ algebra of continuous maps $|K|\to A$, where $|K|$ is the geometric realization of $K$.

Let $A\in\Csep$, and consider the simplicial unit interval $\Delta^1\in\cS_{\fin}$. Then $\hom(\Delta^1,A)$ is just the $C^*$-algebra of continuous maps from the topological unit interval $I=|\Delta^1|$ to $A$. We have two simplicial maps $\Delta^0\to\Delta^1$, taking the values $0$ and $1$. These maps induce two maps in $\Csep$, which we denote
$$\pi_0,\pi_1:\hom(\Delta^1,A)\to \hom(\Delta^0,A)\cong A.$$
These maps are given by evaluation at $0$ and $1$ respectively. There is a unique simplicial map $\Delta^1\to\Delta^0$, which induces a map which we denote
$$\iota:A\cong \hom(\Delta^0,A)\to\hom(\Delta^1,A).$$
This map sends an element to the constant map at that element. The $C^*$ algebra $\hom(\Delta^1,A)$ together with the maps $\pi_0,\pi_1$ and $\iota$ is called the standard path object for $A$ given by the simplicial structure.
\end{defn}

\begin{defn}
We define $\cW$ to be the class of homotopy equivalences in $\Csep$, and $\cF$ to be the class of Schochet fibrations in $\Csep$.
\end{defn}

\begin{prop}[Andersen--Grodal, Uuye]\label{p:SC_WFC}
The triple $({\Csep},{\cW},{\cF})$ is a weak fibration category.
\end{prop}

\begin{proof}
In \cite[Corollary 3.9]{AndGro} it is shown that $({\Csep},{\cW},{\cF})$  is a fibration category in the sense of Baues \cite{Bau}. Since $\Csep$ as finite limits, we obtain that it is also a weak fibration category (see also \cite{Uuye}).

It is worthwhile to describe explicitly a factorization of the morphisms in $\Csep$ into a weak equivalence followed by a fibration. Let $f:A\to B$ be a morphism in $\Csep$. We define $P(f)\in\Csep$ to be the pull back
$$\xymatrix{P(f)\ar[r]\ar[d] & \hom(\Delta^1,B)\ar[d]^{\pi_0} \\
             A \ar[r]^f & B.}$$

We define a morphism $i:A\to P(f)=A\times_{B}\hom(\Delta^1,B)$ to be the one induced by the commutative square
$$\xymatrix{A\ar[d]^{\id}\ar[r] & \hom(\Delta^1,B)\ar[d]^{\pi_0} \\
                 A \ar[r]^{f} & B,}$$
the upper horizontal map being the composite: $A\xrightarrow{f} B\overset{\iota}{\longrightarrow} \hom(\Delta^1,B)$.

We define a morphism $p:P(f)=A\times_{B}\hom(\Delta^1,B)\to B$ to be the composition:
$$P(f)\xrightarrow{}  \hom(\Delta^1,B)\overset{\pi_1} \longrightarrow B.$$

Clearly $f=pi$, and we call this the mapping cylinder factorization.
It is shown in \cite{SchTop3} that the mapping cylinder factorization is a indeed a factorization into a weak equivalence followed by a fibration. We note that this factorization is furthermore \emph{functorial}.
\end{proof}

We now want to show that the weak fibration category $({\Csep},{\cW},{\cF})$ is simplicial. First, it is not hard to see that Definition \ref{d:action} indeed defines a bifunctor
$$\hom(-,-):\cS_{\fin}^{\op}\times\cC\to \cC$$
that commutes with finite limits in every variable separately, and that there are coherent natural isomorphisms
$$\hom(L, \hom(K, A))\cong \hom(K\times L, A),$$
$$\hom(\Delta^0, A)\cong A,$$
for $A$ in $\Csep$ and $K,L$ in $\cS_{\fin}$.

\begin{defn}\label{SC enriched}
Using Definition \ref{d:action}, we can turn $\Csep$ into a category enriched in simplicial sets by defining for every $A,B\in\Csep$ and $n\geq 0$
$$\Map_{\Csep}(A,B)_n:= \Hom_{\Csep}(A,\hom(\Delta^n,B)).$$
\end{defn}
It is not hard to see that for every $K\in \cS_{\fin}$ and $A,B\in \Csep$ we have a natural isomorphism
$$\Map_{\cS}(K,\Map_{\Csep}(A,B))\cong \Map_{\Csep}(A,\hom(K,B)).$$
Since $\Csep$ is enriched in simplicial sets, we can consider the enriched Yoneda embedding
$$Y:A\mapsto \Map_{\Csep}(-,A):\Csep\to \cS^{\Csep^{\op}}.$$

\begin{lem}\label{l:Y commute}
The Yoneda embedding  $Y:\Csep\to \cS^{\Csep^{\op}}$  commutes with finite limits and the simplicial coaction.
\end{lem}

\begin{proof}
The fact that $Y$ commutes with finite limits is clear.
It is left to show that there are coherent natural isomorphisms
$$\Map_{\Csep}(-,\hom(K,A))\cong \hom(K,\Map_{\Csep}(-,A))$$
for $K\in \cS_{\fin}$ and $A\in \Csep$.
Thus, for every $K\in \cS_{\fin}$ and $A,B\in \Csep$ we need to supply an isomorphism
$$\Map_{\Csep}(B,\hom(K,A))\cong \Map_{\cS}(K,\Map_{\Csep}(B,A)),$$
but this is clear.
\end{proof}

\begin{lem}\label{l:enriched Top S}
For every $A,B\in \Csep$ we have a natural isomorphism
$$\Sing(\Csep(B,A))\cong\Map_{\Csep}(B,A),$$
where $\Sing$ denotes the singular simplices functor and $\Csep(A,B)$ is the \emph{space} of $*$-homomorphisms defined in Remark \ref{r:Top enrich}.
\end{lem}

\begin{proof}
By \cite[Lemma 2.4]{Uuye}, for every $n\geq 0$ there is a natural isomorphism
$$\Sing(\Csep(B,A))_n=\Hom_{\Top} (|\Delta^n|,\Csep(B,A))\cong \Hom_{\Csep}(B,\tC(|\Delta^n|,A)) =$$
$$=\Hom_{\Csep}(B,\hom(\Delta^n,A))=\Map_{\Csep}(B,A)_n.$$
From this the result clearly follows.
\end{proof}

\begin{prop}\label{p:transfer}
Consider the  Yoneda embedding  $Y:\Csep\to \cS^{\Csep^{\op}}$, and let $\cS^{\Csep^{\op}}$ be endowed with the \emph{projective} model structure. Let $p: A\to B$ be a map in $\Csep$. Then the following hold:
\begin{enumerate}
\item The map $p$ is a homotopy equivalence in $\Csep$ iff $Y(p)$ is a weak equivalence in $\cS^{\Csep^{\op}}$.
\item The map $p$ is a Schochet fibration in $\Csep$ iff $Y(p)$ is a fibration in $\cS^{\Csep^{\op}}$.
\end{enumerate}
\end{prop}

\begin{proof}\
\begin{enumerate}
\item This follows from Lemma \ref{l:enriched Top S} and \cite[Proposition 1.2.4.1]{Lur}.
\item Suppose that $p$ is a Schochet fibration. Then, by \cite[Proposition 2.18]{Uuye}, for every $D\in\Csep$ the induced map $\Csep(D,A)\to \Csep(D,B)$ is a Serre fibration. It follows from Lemma \ref{l:enriched Top S} and the fact that that the functor $\Sing:\Top\to\cS$ sends Serre fibrations to Kan fibrations that for every $D\in\Csep$ the induced map $\Map_{\Csep}(D,A)\to \Map_{\Csep}(D,B)$ is a Kan fibration. Thus $Y(p)$ is a projective fibration.

     Now suppose that $Y(p)$ is a projective fibration, that is,
     for every $D\in\Csep$ the induced map $\Map_{\Csep}(D,A)\to \Map_{\Csep}(D,B)$ is a Kan fibration. Let $D\in\Csep$, and consider a commutative diagram of the form
$$\xymatrix{
|\Delta^{\{0\}}| \ar[d] \ar[r] & \Csep(D,A)\ar[d]^{p_*}\\
|\Delta^{1}|\ar[r] & \Csep(D,B).}$$
We need to show that there exists a lift $|\Delta^{1}|\map \Csep(D,A)$. Using the fact that we have an adjoint pair
$$|-|:\cS\rightleftarrows\Top:\Sing$$
and Lemma \ref{l:enriched Top S}, we see that it is enough to find a lift in the following diagram:
$$\xymatrix{
\Delta^{\{0\}} \ar[d] \ar[r] & \Map_{\Csep}(D,A)\ar[d]^{p_*}\\
\Delta^{1}\ar[r] & \Map_{\Csep}(D,B).}$$
But such a lift exists since $\Map_{\Csep}(D,A)\to \Map_{\Csep}(D,B)$ is a Kan fibration.
\end{enumerate}
\end{proof}

\begin{prop}\label{p:simplicial}
With the cotensor action of Definition \ref{d:action}, $\Csep$ is a simplicial weak fibration category (see Definition \ref{d:simplicial_weak}).
\end{prop}

\begin{proof}
By what is explained after Proposition \ref{p:SC_WFC}, we only need to show that for every cofibration $j:K\to L$ in $\cS_{\fin}$ and every fibration $p:A\to B$ in $\Csep$, the induced map:
$$\hom(L,A)\to \hom(K,A) \prod_{\hom(K,B)}\hom(L,B)$$
is a fibration (in $\Csep$), which is acyclic if either $j$ or $p$ is.
But this follows from Lemma \ref{l:Y commute}, Proposition \ref{p:transfer} and the fact that the projective model structure on $\cS^{\Csep^{\op}}$ is simplicial (see for example \cite[Remark A.3.3.4]{Lur}).
\end{proof}

\noindent
\subsection{The model structure on $\Pro(\Csep)$.}\label{ss:model pro C}
We now turn to our main theorem.
\begin{thm}\label{t:proC_main}
There exists a simplicial model category structure on $\Pro(\Csep)$ such that:
\begin{enumerate}
\item The weak equivalences are $\mathbf{W} := Lw^{\cong}(\cW)$.
\item The fibrations are $\mathbf{F} := R(Sp^{\cong}(\cF))$.
\item The cofibrations are $\mathbf{C} :={}^{\perp}(\cF\cap \cW)$.
\end{enumerate}

Moreover, this model category is cocombinatorial, with set of generating fibrations $\cF$ and set of generating acyclic fibrations $\cF\cap \cW$.

The model category $\Pro(\Csep)$ has the following further properties:
\begin{enumerate}
\item The acyclic fibrations are given by $\mathbf{F}\cap\mathbf{W} = R(Sp^{\cong}(\cF\cap\cW))$.
\item Every object in $\Pro(\Csep)$ is cofibrant.
\item $\Pro(\Csep)$ is proper.
\end{enumerate}
\end{thm}

\begin{proof}
The triple $({\Csep},{\cW},{\cF})$ is an (essentially) small simplicial weak fibration category by Proposition \ref{p:SC_WFC} and Proposition \ref{p:simplicial}. Thus it remains to show that $\Csep$ satisfies the conditions of Theorem \ref{t:tool_main}.

\begin{enumerate}
\item The fact that $\Csep$ has finite colimits follows from the existence of amalgamated free products in $\Csep$ (see \cite{Ped}).
\item Let $A$ be an object in $\Csep$. We need to show that the map $A\to 0$ is a Schochet fibration. Let $D\in\Csep$ and let
$$\xymatrix{
\text{$\{0\}$} \ar[d] \ar[r]^f & \Csep(D,A)\ar[d]^{p_*}\\
[0,1]\ar[r] & \Csep(D,0)\cong *}$$
be a commutative diagram. We can define a lift $[0,1]\map \Csep(D,A)$ in the diagram above to be the constant map at $f(0)$, so we are done.
\item By Lemma \ref{l:enriched Top S}, a map in $\Csep$ is a weak equivalence iff that is a homotopy equivalence in the simplicial category $\Csep$.
\item The fact that every map in $\cF\cap\cW$ admits a section is shown in \cite[Proposition 1.13 (a)]{SchTop3}.
\end{enumerate}
\end{proof}

\begin{rem}\label{r:proC_main_dual}
By Theorem \ref{t:proC_main} and Remark \ref{r:tool_main}, there exists a simplicial model category structure on $\Pro(\Csep)^{\op}\cong\Ind(\Csep^{\op})$, given by $(\Ind(\Csep^{\op}),\mathbf{W}^{\op},\mathbf{C}^{\op},\mathbf{F}^{\op})$,
with the following properties:
\begin{enumerate}
\item The model category $\Ind(\Csep^{\op})$ is combinatorial, with set of generating cofibrations $\cF^{\op}$ and set of generating acyclic cofibrations $\cF^{\op}\cap \cW^{\op}$.
\item Every object in $\Ind(\Csep^{\op})$ is fibrant.
\item $\Ind(\Csep^{\op})$ is proper.
\end{enumerate}
We note that the pro picture is more in accordance with the tradition in noncommutative geometry, while the dual ind picture is more in alliance with the conventions in homotopy theory. We will thus be using both pictures, at our convenience, throughout the paper.
\end{rem}

\subsection{Relation to the $\infty$-category $\Pro(\Csepinf)$.}\label{ss:relation infty}
As we explained in the introduction, there are two natural ways of considering $\Csep$ as an $\infty$-category. One is to consider $\Csep$ as a topologically enriched category and take its topological nerve as in \cite[Section 1.1.5]{Lur}. This approach was taken by the third author in \cite{MahNSH}, where the opposite of this $\infty$-category was called ``the  $\infty$-category of pointed compact metrizable noncommutatives spaces". The other approach is to take the $\infty$-localization of $\Csep$ as a weak fibration category. (See Appendix \ref{ss:infinity} for  the definition of the $\infty$-localization of a relative category.) This $\infty$-category will be denoted by $\Csepinf$. In Proposition \ref{p:equiv_infinity} we show that these two ways are equivalent. 

\begin{rem}
The notation used by the third author in \cite{MahNSH} to denote the $\infty$-category of pointed compact metrizable noncommutative spaces is $(\Csepinf)^\op$, so by Proposition \ref{p:equiv_infinity} we have no ambiguity of notation.
\end{rem}

\begin{prop}\label{p:equiv_infinity}
The $\infty$-category $(\Csepinf)^\op$ is naturally equivalent to the $\infty$-category of pointed compact metrizable noncommutatives spaces defined in \cite{MahNSH}.
\end{prop}

\begin{proof}
Recall that $\Csep$ is a simplicial weak fibration category. The simplicial coaction is given by:
$$\hom(K, A):=\tC(|K|)\otimes A=\tC(|K|,A)\in \Csep^\op,$$
for $K$ in $\cS_{\fin}$ and $A$ in $\Csep$.
By Theorem \ref{t:proC_main}, $\Pro(\Csep)$ is a simplicial model category where for every $K$ in $\cS_{\fin}$ and $\{A_t\}_{t\in T}$ in $\Pro(\Csep^\op)$ the simplicial coaction is just objectwise
$$\hom(K, \{A_t\}_{t\in T})\cong \{\hom(K, A_t)\}_{t\in T}.$$
Furthermore, for every $A= \{A_t\}_{t\in T}$ and $B= \{B_s\}_{s\in S}$ in $\Pro(\Csep)$ the simplicial enrichment is given by
$$\Map_{\Pro(\Csep)}(A,B)_n=\Hom_{\Pro(\Csep)}( A,\hom(\Delta^n,B)),$$
for $n\geq 0$.
In particular, if $A$ and $B$ belong to $\Csep$ the simplicial enrichment is given by
$$\Map_{\Pro(\Csep)}(A,B)_n=\Hom_{\Csep}( A,\hom(\Delta^n,B))=\Map_{\Csep}(A,B)_n,$$
for $n\geq 0$.
By Lemma \ref{l:enriched Top S} we have a natural isomorphism
$$\Sing(\Csep(A,B))\cong\Map_{\Csep}(A,B).$$
By Theorem \ref{t:BarHarHor} we deduce that the natural functor
$$\Csepinf\to\Pro(\Csep)_\infty$$
is derived fully faithful.  Furthermore, since $\Pro(\Csep)$ is a simplicial model category and every object in $\Csep$ is both fibrant and cofibrant in $\Pro(\Csep)$, we have that
$$\Map_{\Pro(\Csep)_\infty}(A,B)\simeq\Map_{\Pro(\Csep)}(A,B).$$
Thus we obtain
$$\Map_{\Csepinf}(A,B)\simeq\Sing(\Csep(A,B)).$$
\end{proof}

In \cite{MahNSH}, the $\infty$-category $\Ne\cS_*$ of pointed noncommutative spaces (that are not necessarily compact) was defined to be the ind-category of the $\infty$-category of pointed compact metrizable noncommutatives spaces. We thus obtain the following proposition:

\begin{prop}\label{p:noncomm_space}
We have a natural equivalence of $\infty$-categories
$$\Ne\cS_*\simeq(\Ind(\Csep^\op))_\infty$$
(see Remark \ref{r:proC_main_dual}).
\end{prop}

\begin{proof}
In light of Proposition \ref{p:equiv_infinity} we have that
$$\Ne\cS_*\simeq\Ind((\Csep^\op)_\infty).$$
By Remark \ref{r:BarHarHor} we have a natural equivalence of $\infty$-categories
$$\Ind((\Csep^\op)_\infty)\simeq(\Ind(\Csep^\op))_\infty.$$
\end{proof}

As we have explained in Remark \ref{r:proC_main_dual}, $\Ind(\Csep^\op)$ is a proper combinatorial pointed simplicial model category, and the domains of the generating cofibrations of $\Ind(\Csep^\op)$ can be taken to be cofibrant (note that every object in $\Csep^\op$ is cofibrant in $\Ind(\Csep^\op)$). Thus, as explained in Appendix \ref{s:stab_model}, we can construct the stable left proper combinatorial simplicial model category $\SpN(\Ind(\Csep^\op))$, together with the natural simplicial left Quillen functor
$$G_0:\Ind(\Csep^\op)\to\SpN(\Ind(\Csep^\op)).$$

\begin{prop}\label{p:N_spectra}
The $\infty$-category $\SpN(\Ind(\Csep^\op))_\infty$ is naturally equivalent to the $\infty$-category $\Sp(\Ne\cS_*)$ considered in \cite{MahNSH}, and
$$\mathbb{L}G_0:\Ind(\Csep^\op)_\infty\to \SpN(\Ind(\Csep^\op))_\infty$$
is equivalent to
$$\Sigma^{\infty}:\Ne\cS_*\to \Sp(\Ne\cS_*)$$
under this natural equivalence.
\end{prop}

\begin{proof}
As explained in Appendix \ref{s:stab_infinity}, we have a natural equivalence of $\infty$-categories
$$\SpN(\Ind(\Csep^\op))_\infty\simeq \Sp(\Ind(\Csep^\op)_\infty),$$
and
$$\mathbb{L}G_0:\Ind(\Csep^\op)_\infty\to \SpN(\Ind(\Csep^\op))_\infty$$
is equivalent to
$$\Sigma^{\infty}_{\Ind(\Csep^\op)_\infty}:\Ind(\Csep^\op)_\infty\to \Sp(\Ind(\Csep^\op)_\infty)$$
under this natural equivalence. Combining this with Proposition \ref{p:noncomm_space} we obtain the desired result.
\end{proof}

\subsection{Relation of the category $\Pro(\Csep)$ to pro-$C^*$-algebras}\label{s:connect}
In this section we connect the underlying category of our model structure,  $\Pro(\Csep)$, with the more classical notion of pro-$C^*$-algebras, namely, topological $*$-algebras that are cofiltered limits of (separable) $C^*$-algebras (see \cite{Phi},\cite{Phi2}).

Let $\TPro(\Csep)$ denote the full subcategory of the category of topological $*$-algebras, spanned by those objects which are cofiltered limits of objects in $\Csep$ (the limit is taken in the category of topological $*$-algebras).

Let $A$ be an object in $\TPro(\Csep)$. Let $S(A)$ denote the set of all continuous $C^*$-seminorms on $A$. We regard $S(A)$ as a directed set in the obvious way. Namely, given $p,q\in S(A)$ we say $q \ge p$ if $q(a) \ge p(a)$ holds for all $a\in A$.

We define $\cL(A)\in \Pro(\Cgen)$ to be the diagram $\cL(A):S(A) \to \Cgen$ which sends $p\in S(A)$ to the $C^*$-algebra given by $A/\ker p$, with the norm inherited from $p$ (the fact that this is indeed a $C^*$-algebra is shown in \cite[Corollary 1.12]{Phi2}).
\begin{lem}\label{l:separable}
For every $p\in S(A)$ the $C^*$-algebra $\cL(A)_p$ is separable.
\end{lem}

\begin{proof}
We adapt arguments of Phillips from \cite[on page 131]{Phi}. Since $A$ is in $\TPro(\Csep)$ we can assume $A = \lim_{i \in I} X_{i}$ for some cofiltered system $X:I \to \Csep$. For each $i\in I$ let $p_i\in S(A)$ denote the pullback of the $C^*$-norm of $X_i$ along the natural map $A \to X_i$. Now pick any $p \in S(A)$. Since the $C^*$-seminorms $p_i$ determine the topology of $A$ there is a constant $c>0$ such that $p(a) \le c \max\{p_{i_1}(a),...,p_{i_r}(a)\}$ for all $a\in A$. The indexing set $I$ is cofiltered, therefore there is an index $i$ such that $p_i\ge p_{i_j}$ for all $1 \le j \le r$. Hence we have a quotient map $\cL(A)_{p_i} \to \cL(A)_{p}$. But $\cL(A)_{p_i}$ is isomorphic to a sub $C^*$-algebra of $X_i$, namely, the image of the natural map $A\to X_i$. Since $X_i$ is separable, we conclude that $\cL(A)_p$ is also separable.
\end{proof}

Thus, we have defined an object $\cL(A)\in\Pro(\Csep)$. In fact, it is not hard to see that we actually obtain a functor
$$\cL:\TPro(\Csep)\to\Pro(\Csep).$$

We now define a functor in the other direction:
\begin{defn}
Consider the inclusion $i:\Csep\hookrightarrow \TPro(\Csep)$. Since the category $\TPro(\Csep)$ has cofiltered limits, it follows from the universal property of the $\Pro$ construction that $i$ can be extended naturally to a functor $\lim:\Pro(\Csep)\to \TPro(\Csep)$ that commutes with cofiltered limits. We call this functor $\lim$ since it is indeed given by taking the limit in $\TPro(\Csep)$ of the input diagram.
\end{defn}

\begin{prop}\label{p:pro adjoint}
The above defined functors form an adjoint pair:
$$\cL:\TPro(\Csep)\rightleftarrows \Pro(\Csep):\lim,$$
where the unit $\id\to \lim\circ\cL$ is a natural isomorphism.
\end{prop}

\begin{proof}
Let $A$ be an object in $\TPro(\Csep)$ and let $\{X_i\}_{i\in I}$ be an object in $\Pro(\Csep)$. We need to show that there is a natural bijection
$$\Hom_{\TPro(\Csep)}(A,\lim_{i\in I}X_i)\cong \Hom_{\Pro(\Csep)}(\cL(A),\{X_i\}_{i\in I}).$$
We thus need to show that
$$\lim_{i\in I}\Hom_{\TPro(\Csep)}(A,X_i)\cong \lim_{i\in I}\Hom_{\Pro(\Csep)}(\cL(A),X_i).$$
It follows that it is enough to show that for any object $X$ in $\Csep$ we have
$$\Hom_{\TPro(\Csep)}(A,X)\cong \Hom_{\Pro(\Csep)}(\cL(A),X)\cong \colim_{p\in S(A)}\Hom_{\Csep}(\cL(A)_p,X),$$
but this follows from \cite[Lemma 1.1.5]{Phi}.
Now the unit $\id\to \lim\circ\cL$ of this adjunction is a natural $*$-isomorphism by \cite[Proposition 1.2]{Phi2} (citing \cite{Schm}).
\end{proof}

\begin{cor}\label{c:onto}
The left adjoint $\cL:\TPro(\Csep)\to \Pro(\Csep)$ is fully faithful. It follows that the functor
$$\cR:=\cL\circ\lim:\Pro(\Csep)\to\Pro(\Csep)$$
is a colocalization functor (a coreflector). Let us denote by $\cL\TPro(\Csep)$ the essential image of $\TPro(\Csep)$ under $\cL$. Then $\cL\TPro(\Csep)$, which is equivalent to $\TPro(\Csep)$, is a coreflective full subcategory of $\Pro(\Csep)$ and we have an adjoint pair
$$i:\cL\TPro(\Csep)\rightleftarrows \Pro(\Csep):\cR,$$
where $i$ is the inclusion.
\end{cor}
\begin{rem}
It might be possible to transfer our model structure on $\Pro(\Csep)$ through the adjunction
$$\cL:\TPro(\Csep)\rightleftarrows \Pro(\Csep):\lim,$$
and obtain a model structure also on $\TPro(\Csep)$.
\end{rem}


\section{Triangulated homology theories on $\Ind(\Csep^\op)$}\label{s:homology}
\subsection{Definition of triangulated homology theories}\label{ss:define homology}
In this subsection we define the notion of triangulated homology theories on a pointed cocomplete $\infty$-category.

Let $\cC$ be a pointed finitely cocomplete $\infty$-category. A diagram
$$X_0\map X_1\map X_2\map \cdots$$
in $\cC$ is called a {\em cofiber sequence} if each $X_{i+2}$ is the cofiber of the previous map $X_i\to X_{i+1}$. Thus, a cofiber sequence is completely determined, up to equivalence, by the first map $X_0\map X_1$. Note that if $\cC= \cM_\infty$, where $\cM$ is a pointed \emph{model category}, a cofiber sequence can be calculated using homotopy colimits, that is, by turning each map into a cofibration and then taking the cofiber in the underlying pointed category of $\cM$.

Let $A\map B\map C\map D$ be a cofiber sequence in $\cC$. Then we have the following diagram of pushout squares:
$$\xymatrix{A\ar[r]\ar[d] & B\ar[r]\ar[d] & \ast\ar[d]\\
              \ast\ar[r] &  C \ar[r]  & D.}$$
It follows that $D\simeq\Sigma A$ (see Appendix \ref{s:stab_infinity}). Thus, every cofiber sequence in $\cC$ has the form
$$A\to B\to C\to \Sigma A\to \Sigma B\to \Sigma C\to \Sigma^2 A\to \cdots.$$

The following definition is motivated by \cite{Quillen,HovBook,SchTop3,Tho}:

\begin{defn}\label{d:homology theory}
Let $\cC$ be a pointed cocomplete $\infty$-category.
\begin{enumerate}
\item Let $\mathcal{T}$ be a triangulated category with coproducts. A {\em triangulated homology theory} on $\cC$, with values in $\mathcal{T}$, is a pointed functor $\H:\Ho\cC\map\mathcal{T}$ such that:
 \begin{itemize}
 \item For any cofiber sequence in $\cC$ of the form $A\to B\to C\to \Sigma A$ the diagram $\H(A) \to \H(B) \to \H(C) \to \H(\Sigma A)$ is a distinguished triangle in $\mathcal{T}$.
 \item $\H$ preserves coproducts.
 \end{itemize}

\item
A cohomology theory on $\cC$ is a $\ZZ$ indexed sequence of pointed functors $\H^n:\Ho\cC^{\op}\to \Ab$ together with natural isomorphisms $\H^n\cong \H^{n+1}\circ \Sigma$ such that:
 \begin{itemize}
 \item For any cofiber sequence in $\cC$ of the form $A\to B\to C\to \Sigma A$ the diagram $\H^n(C) \to \H^n(B) \to \H^n(A) $
      is exact.
 \item $\H^n$ preserves products.
 \end{itemize}
\end{enumerate}
If $\cM$ is a pointed model category, then we define a triangulated homology theory or a cohomology theory on $\cM$ to be a triangulated homology theory or a cohomology theory on $\cM_\infty$. (It is shown in \cite{BarHarHor} that $\cM_\infty$ is indeed a pointed cocomplete $\infty$-category.)
\end{defn}

Note that any triangulated homology theory $\H:\Ho\cC\map\mathcal{T}$ and any object $S$ in $\mathcal{T}$ give rise to a cohomology theory $$\H^n:=\Hom_{\mathcal{T}}(\Sigma^{-n}\circ \H(-),S):\Ho\cC^{\op}\to \mathtt{Ab},$$ 
where $\Sigma$ here denotes the suspension functor in $\cT$ (see \cite{Nee} Example 1.1.13).

\begin{rem}\
\begin{enumerate}
\item Note that if $\cC$ is a cocomplete $\infty$-category then $\Ho\cC$ admits arbitrary coproducts, and they can be calculated as coproducts in $\cC$.
\item The original definition of a triangulated homology theory on $\Csep$ appeared in \cite{Tho}. After a suitable reversal of arrows our definition, when applied to $(\Ind(\Csep^\op))_\infty$, is a little more general than the original one when restricted to $\Csep$. It is similar in spirit to {\em cofiber homology theories} in \cite{SchTop3}.
\end{enumerate}
\end{rem}

Recall from Appendix \ref{s:stab_infinity} that a pointed $\infty$-category with finite colimits is called stable if the suspension functor $\Sigma:\cC\to \cC$ is an equivalence. If $\cC$ is a stable $\infty$-category then it is shown in \cite{Lur2} that $\Ho\cC$ is naturally a triangulated category; the suspension functor in $\Ho\cC$ is the one induced by $\Sigma:\cC\to \cC$, and the distinguished triangles are given by the cofiber sequences in $\cC$, after projection to $\Ho\cC$.

Let $\cC$ be a stable $\infty$-category with small colimits. Note that a cohomology theory $\H^n$ on $\cC$ is entirely determined by $\H^0$ since we have natural isomorphisms $\H^n\cong \H^0(\Sigma^{-n}(-))$.
Thus we see that if $\mathcal{T}$ is a triangulated category with coproducts then a triangulated homology theory on $\cC$, with values in $\mathcal{T}$, is just a triangulated coproduct preserving functor $\H:\Ho\cC\map\mathcal{T}$, and a cohomology theory on $\cC$ is just a decent cohomological functor $\H^0:\Ho\cC^{\op}\to \mathtt{Ab}$ (we refer to \cite{Nee} for the terminology concerning triangulated categories).
The following lemma is straightforward:
\begin{lem}\label{l:homology_theory}
Let $\cC$ and $\cD$ be pointed cocomplete $\infty$-categories and suppose that $\cD$ is stable. Let $F:\cC\to \cD$ be a colimit preserving functor. Then $\Ho\cD$ is naturally a triangulated category with coproducts and
$$\Ho F:\Ho\cC\to\Ho\cD$$
is a triangulated homology theory on $\cC$.
\end{lem}

Now suppose that $\cC$ is a pointed presentable $\infty$-category. Then we have a natural choice of a triangulated homology theory on $\cC$. Namely, as explained in Appendix \ref{s:stab_infinity}, we can construct a stable presentable $\infty$-category $\Sp(\cC)$, together with a left adjoint
$$\Sigma^\infty:\cC\to\Sp(\cC).$$ If $\cC$ were the $\infty$-category of pointed spaces, then $\Sp(\cC)$ would correspond to the stable $\infty$-category of spectra.
Then according to Lemma \ref{l:homology_theory}, $\Ho\Sp(\cC)$ is naturally a triangulated category with coproducts and
$$\Ho \Sigma^\infty:\Ho\cC\to\Ho\Sp(\cC)$$
is a triangulated homology theory on $\cC$.

\begin{rem}\label{r:Brown}
Since $\Ne\cS_*\simeq\Ind(\Csepinf^\op)$ is a compactly generated $\infty$-category it follows from \cite[Proposition 1.4.3.7]{Lur2} that $\Sp(\Ne\cS_*)$ is a stable compactly generated $\infty$-category. Thus, $\Ho\Sp(\Ne\cS_*)$ is a compactly generated triangulated category (see \cite[Proposition 1.4.4.1]{Lur2}). It follows from \cite[Theorem 8.3.3]{Nee} that any cohomology theory on $\Sp(\Ne\cS_*)$ is representable. That is, if $\H^0:\Ho\Sp(\Ne\cS_*)^{\op}\to \mathtt{Ab}$ is a cohomology theory on $\Sp(\Ne\cS_*)$ then there exists an object $T$ in $\Ho\Sp(\Ne\cS_*)$ such that $\H^0$ is naturally isomorphic to
$$\Ho\Sp(\Ne\cS_*)(-,T).$$
This is called Brown representability. We will define in the next subsection several cohomology theories on $\Sp(\Ne\cS_*)$, but we will not need this result since we will be able to give a rather explicit description of a representing object.
\end{rem}

\subsection{Triangulated homology theories on $\Ind(\Csep^\op)$}\label{ss:THT}
In this subsection we construct several triangulated homology theories on $\Ind(\Csep^\op)$ (see Remark \ref{r:proC_main_dual}). By taking the opposite category they become bivariant homology theories that are applicable to {\em all} projective systems of separable $C^*$-algebras. Using Proposition \ref{p:noncomm_space} we see that a triangulated homology theory on $\Ind(\Csep^\op)$ is equivalent to a triangulated homology theory on $\Ind(\Csep^\op)_\infty\simeq \Ne\cS_*$. We thus recall a construction defined by the third author in \cite{MahNSH}, which associates a triangulated homology theory on $\Ne\cS_*$ to any set of morphisms in $\Csep^\op$. Due to certain improvements we incorporate in it, and for the convenience of the reader, we bring a detailed account of this construction here (see Theorem \ref{p:homology_main}). We then transform this construction, which uses the language of $\infty$-categories, to the world of model categories (see Theorem \ref{p:homology_main_model}). We end by considering several examples of this general construction.

Note, that the simplicial model category $\Ind(\Csep^\op)$ is pointed, where the zero object in $\Ind(\Csep^\op)$ is just the zero $C^*$ algebra. Thus the $\infty$-category $(\Ind(\Csep^\op))_\infty$ is indeed a pointed cocomplete $\infty$-category. If $(K,x)$ in a pointed finite simplicial set and $A\in\Csep^\op$, then the smash product $K\wedge A\in \Csep^\op$ is just the separable $C^*$-algebra of \emph{pointed} continuous maps from $(|K|,x)$ to $(A,0)$.

We begin with a small introduction.
As was shown in \cite{MahNSH}, and also follows easily from the model structure constructed here, the $\infty$-category $\Csepinf^\op$ admits finite colimits. Thus, as explained in Appendix \ref{s:stab_infinity}, there is a natural equivalence of $\infty$-categories
$$\Sp(\Ne\cS_*)\simeq\Sp(\Ind(\Csepinf^\op))\simeq \Ind(\SW(\Csepinf^\op)).$$
Recall that the objects
of the stable $\infty$-category $\SW(\Csepinf^\op)$ are pairs $(A,n)$ where $A\in\Csepinf^\op$ and $n\in \mathbb{N}$, and the mapping spaces
are given by
$$\Map_{\SW(\Csepinf^\op)}((A,n),(B,m)) = \colim_k\Map_{\Csepinf^\op}(\Sigma^{k-n} A,\Sigma^{k-m} B).$$ If $\cC$ were the $\infty$-category of finite pointed spaces, then $\SW(\cC)$ would corresponed to the $\infty$-categorical analogue of the Spanier--Whitehead category of finite spectra. We denote the natural map from $\Csepinf^\op$ to $\SW(\Csepinf^\op)$ by $\Sigma^\infty$. Note that this is indeed the restriction of $\Sigma^\infty:\Ne\cS_*\to\Sp(\Ne\cS_*)$ to $\Csepinf^\op$.

Using the fact that $\pi_0$ commutes with filtered homotopy colimits of simplicial sets, we see that $\Ho \SW(\Csepinf^\op)$ is equivalent to the triangulated category denoted $\Ho\Csep[\Sigma^{-1}]^\op$ in \cite{MahNSH} (obtained from $\Ho\Csep^\op$ by inverting the endofunctor $\Sigma$). In particular we see that we have a natural fully faithful inclusion of triangulated categories
$$\Ho\Csep[\Sigma^{-1}]^\op\hookrightarrow\Ho\Sp(\Ne\cS_*).$$
It is not hard to see that for any set of morphisms $R$ in $\SW(\Csepinf^\op)$, the set of objects $\{\cone(g)\;|\;g\in R\}$ in $\SW(\Csepinf^\op)$ is the same as
$$\{Z\;|\;\text{$\exists$ triangle in $\Ho \SW(\Csepinf^\op)$ of the form $X\xrightarrow{g}Y\to Z\to \Sigma X$ with $g\in R$}\}.$$

We now invoke a construction used in \cite{MahNSH}.
\begin{prop}[Mahanta]\label{p:homology_main}
Let $S$ be a set of morphisms in $\Csep^\op$ and let $\cA_S$ denote the smallest stable $\infty$-subcategory of $\SW(\Csepinf^\op)$ containing the set of objects $\{\cone(\Sigma^{\infty}g)\;|\;g\in S\}$. We define $$\cH_S:=\Ind(\SW(\Csepinf^\op)/\cA_S),$$
where $\SW(\Csepinf^\op)/\cA_S$ is the cofiber of the inclusion $\cA_S\hookrightarrow\SW(\Csepinf^\op)$ in $\Catex$ (see Appendix \ref{s:stab_infinity}).
Then $\cH_S$ is a compactly generated stable $\infty$-category. Moreover, we have the following:
\begin{enumerate}
 \item There is a localization functor $L:\Sp(\Ne\cS_*)\to \cH_S,$
which after composing with $\Sigma^\infty$ $$\Sigma^\infty_S:=L\circ\Sigma^\infty:\Ne\cS_*\to \cH_S$$
induces a triangulated homology theory on $\Ne\cS_*$ $$\Ho\Sigma^\infty_S:\Ho\Ne\cS_*\to \Ho\cH_S.$$

\item There is a canonical fully faithful exact functor of triangulated categories
$$\Ho\Csep[\Sigma^{-1}]^\op/\langle\{\cone(\Sigma^{\infty}g)\;|\;g\in S\}\rangle\hookrightarrow\Ho\cH_S,$$
where the quotient above is Verdier localization.
\end{enumerate}
\end{prop}

\begin{rem}
See \cite[Section 5.5.4]{Lur} for the general theory of localizations of $\infty$-categories.
\end{rem}

\begin{proof}
The ind-category $\Ind(\cA_S)$ is a stable $\infty$-subcategory of $\Ind(\SW(\Csepinf^\op))$. Since the inclusion $\cA_S\hookrightarrow \SW(\Csepinf^\op)$ preserves finite colimits, it follows that the inclusion $\Ind(\cA_S)\to\Ind(\SW(\Csepinf^\op))$ admits a right adjoint. Thus this inclusion is a morphism in $\Prlex$ (see Appendix \ref{s:stab_infinity}) and we define
$$\cH_S:=\Sp(\Ne\cS_*)/\Ind(\cA_S)$$
to be its cofiber in $\Prlex$ (see \cite[Definition 5.4]{BGT}). By \cite[Proposition 5.6]{BGT} we have that the natural functor
$$L:\Sp(\Ne\cS_*)\to\cH_S$$
is a localization. Furthermore, since $\Ind$ preserves cofibers, we have a natural equivalence
$$\cH_S=\Sp(\Ne\cS_*)/\Ind(\cA_S)\simeq \Ind(\SW(\Csepinf^\op))/\Ind(\cA_S)\simeq \Ind(\SW(\Csepinf^\op)/\cA_S).$$
Thus $\cH_S$ is compactly generated and, in particular, accessible.

By \cite[Proposition 5.14]{BGT}, we have a natural equivalence of categories
$$\Ho \SW(\Csepinf^\op)/\Ho\cA_S\simeq\Ho(\SW(\Csepinf^\op)/\cA_S),$$
where the first quotient is the Verdier localization of the triangulated category $\Ho \SW(\Csepinf^\op)$.
It is also easy to see that the triangulated subcategory $\Ho\cA_S$ of $\Ho \SW(\Csepinf^\op)$ is the smallest triangulated subcategory of $\Ho \SW(\Csepinf^\op)$ containing $\{\cone(\Sigma^{\infty}g)\;|\;g\in S\}$, or in other words
$$\Ho\cA_S=\langle\{\cone(\Sigma^{\infty}g)\;|\;g\in S\}\rangle.$$
Thus we obtain
$$\Ho\Csep[\Sigma^{-1}]^\op/\langle\{\cone(\Sigma^{\infty}g)\;|\;g\in S\}\rangle\simeq\Ho(\SW(\Csepinf^\op)/\cA_S),$$
and we see that we have a natural fully faithful inclusion of triangulated categories
$$\Ho\Csep[\Sigma^{-1}]^\op/\langle\{\cone(\Sigma^{\infty}g)\;|\;g\in S\}\rangle\hookrightarrow\Ho\cH_S.$$
\end{proof}

Let $S$ be a set of morphisms in $\Csep^\op$.
For $A,B\in\Ho\Ne\cS_*$, we define
$$(\mathbb{H}_S)_0(A,B) := \Ho \cH_S(\Ho\Sigma^\infty_S(A),\Ho\Sigma^\infty_S(B))\in \Ab.$$
We may extend the $\mathbb{H}_S$-theory to a graded theory as follows:
\beqn
(\mathbb{H}_S)_n(A,B):=\begin{cases}
                                  (\mathbb{H}_S)_0(A,\Sigma^{-n}B) & \text{ if $n< 0$,}\\
                                  (\mathbb{H}_S)_0(\Sigma^nA,B) & \text{ if $n\geq 0$.}
                                 \end{cases}
\eeqn

Since we have a specific model for the $\infty$-category $\Ne\cS_*$, namely $\Ind(\Csep^\op)$, we can also perform the localization described in Proposition \ref{p:homology_main} in the world of model categories. This gives specific models for the localized $\infty$-categories described in Proposition \ref{p:homology_main} and studied by the third author in \cite{MahNSH} and other papers. In particular we obtain models for the stable $\infty$-category of noncommutative spectra $\mathtt{NSp}$ (resp. $\mathtt{NSp'}$) that was constructed in \cite{MahNSH} (resp. \cite{MahColoc}).

\begin{prop}\label{p:homology_main_model}
Let $S$ be a set of morphisms in $\Csep^\op$.
Then there exists a small set of morphisms $T=T_S$ in $\Ho\SpN(\Ind(\Csep^\op))$ such that (in the notation of Proposition \ref{p:homology_main}) $\cH_S$ is modeled by the left Bousfield localization of $\SpN(\Ind(\Csep^\op))$ with respect to $T$, or in other words
$$\cH_S\simeq \mathrm{L}_T \SpN(\Ind(\Csep^\op))_\infty.$$
Furthermore, the left Quillen functor
$$\id:\SpN(\Ind(\Csep^\op))\to\mathrm{L}_T \SpN(\Ind(\Csep^\op))$$
gives rise to an $\infty$-functor
$$\mathbb{L}\id:\SpN(\Ind(\Csep^\op))_\infty\to\mathrm{L}_T \SpN(\Ind(\Csep^\op))_\infty$$
which is equivalent to $L$. The  model category
$\mathrm{L}_T \SpN(\Ind(\Csep^\op))$ is moreover stable simplicial left proper and  combinatorial.
\end{prop}

\begin{rem}
It can be shown that our desired set $T=T_S$ can be taken to be
$$T_S=\{\Sigma^n \mathbb{L}G_0(f)\;|\; f\in S, \:n\in\mathbb{Z}\},$$
but we will not need this result in this paper.
\end{rem}

\begin{proof}
Recall from Proposition \ref{p:N_spectra} that we have a natural equivalence of $\infty$-categories
$$\Sp(\Ne\cS_*)\simeq\SpN(\Ind(\Csep^\op))_\infty.$$
By Theorem \ref{t:spectra}, $\SpN(\Ind(\Csep^\op))$ is a left proper combinatorial simplicial model category. It follows from \cite[Proposition A.3.7.8]{Lur} that every accessible localization of $\Sp(\Ne\cS_*)$ can be manifested by a left Bousfield localization of the model category $\SpN(\Ind(\Csep^\op))$ with respect to a small set of morphisms. In particular, there exists a small set of morphisms $T=T_S$ in $\Ho\SpN(\Ind(\Csep^\op))$ such that
$$\cH_S\simeq \mathrm{L}_T \SpN(\Ind(\Csep^\op))_\infty.$$
Furthermore, the left Quillen functor
$$\id:\SpN(\Ind(\Csep^\op))\to\mathrm{L}_T \SpN(\Ind(\Csep^\op))$$
gives rise to an $\infty$-functor
$$\mathbb{L}\id:\SpN(\Ind(\Csep^\op))_\infty\to\mathrm{L}_T \SpN(\Ind(\Csep^\op))_\infty$$
which is equivalent to $L$.
It follows in particular, that an object $X\in \mathrm{L}_T \SpN(\Ind(\Csep^\op))$ is fibrant if and only if it is fibrant in $\SpN(\Ind(\Csep^\op))$ and the associated object in
$$\SpN(\Ind(\Csep^\op))_\infty\simeq\Sp(\Ne\cS_*)$$
belongs to the full subcategory $\cH_S$.
By Proposition \ref{p:bousfield} and Theorem \ref{t:bousfield}  we know that $\mathrm{L}_T \SpN(\Ind(\Csep^\op))$ is a simplicial left proper combinatorial model category.
\end{proof}

\subsection{Examples}\label{ss:homology examples}
In this subsection, we consider several applications of Propositions \ref{p:homology_main} and \ref{p:homology_main_model}, for different sets of morphisms $S$.

We begin with a general construction.
Let $\phi:B\map C$ be a morphism in $\Csep$. We denote by $\hfib(\phi)$ the pullback in $\Csep$:
$$\xymatrix{
\hfib(\phi)\ar[dd]\ar[rrd] & B \ar[dd]^\phi\ar[dr]^i & \\
 & & P(\phi),\ar[dl]^p\\
0\ar[r] & C. & }$$
where $P(\phi),i,p$ are as in the proof of Proposition \ref{p:SC_WFC}.
Note that there is an induced map
$$\theta(\phi):\fib(\phi)\to \hfib(\phi)$$
in $\Csep$ (where $\fib(\phi)$ denotes the fiber of $\phi$, that is, the kernel of $\phi$).

Consider the following set of morphisms in $\Csep^\op$:
$$S_1 := \{\theta(\phi)^\op\;|\; 0\map A\map B\overset{\phi}{\map} C\map 0 \text{ is a cpc-split extension in }\Csep\}.$$

Fix a minimal projection $p\in\cpt$, where $\cpt$ is the $C^*$ algebra of compact operators on a separable Hilbert space. For any $A\in\Csep$, there is an induced morphism $\iota_A:A\map A\otimes\cpt$ in $\Csep$, sending $a$ to $a\otimes p$. We define another set of morphisms in $\Csep^\op$
$$S_2:= \{\iota_A^\op\;|\; A\in\Csep\}.$$

We now apply Proposition \ref{p:homology_main_model} to the set $S:=S_1\cup S_2$.
We denote the left Bousfield localization of $\SpN(\Ind(\Csep^\op))$ with respect to $T=T_S$ (as in Proposition \ref{p:homology_main_model}) by
$$\mathbf{KK}^{\ind}:=\mathrm{L}_T \SpN(\Ind(\Csep^\op)).$$
Then
$\mathbf{KK}^{\ind}$ is a stable simplicial left proper combinatorial model category, and the $\infty$-category $\mathbf{KK}^{\ind}_\infty$ is compactly generated.
We also denote
$$\mathbf{KK}^{\pro}:=(\mathbf{KK}^{\ind})^{\op},$$
and
$$\KKI:=\Ho\mathbf{KK}^{\ind},\:\:\KKP:=\Ho\mathbf{KK}^{\pro}.$$
We have a composite left Quillen functor
$$\pi_\K:\Ind(\Csep^\op)\xrightarrow{G_0}\SpN(\Ind(\Csep^\op))\xrightarrow{\id}\mathbf{KK}^{\ind},$$
and its left derived functor
$$\LPK:\Ho\Ind(\Csep^\op)\xrightarrow{}\KKI,$$
is a triangulated homology theory on $\Ind(\Csep^\op).$

It is well known that
$$\Ho\Csep[\Sigma^{-1}]^\op/\langle\{\cone(g)\;|\;g\in S\}\rangle$$
is equivalent to the opposite of Kasparov's bivariant $\K$-theory category. (An analogous result for $\E$-theory is shown in \cite{Tho}.) Thus we get that the opposite of Kasparov's bivariant $\K$-theory category is equivalent to the triangulated subcategory generated by the image of the composite functor
$$\Ho\Csep^{\op}\map\Ho\Ind(\Csep^{\op})\xrightarrow{\LL\pi_\K}\KKI.$$

For $A,B\in\Ho\Ind(\Csep^\op)$ we define
$$\KK^\ind_0(A,B) :=(\mathbb{H}_S)_0(A,B)= \KKI(\LPK(A),\LPK(B))\in \Ab,$$
and for any $n\in\mathbb{Z}$ we define
$$
\KK^\ind_n(A,B):=(\mathbb{H}_S)_n(A,B)=\begin{cases}
                                  \KK^\ind_0(A,\Sigma^{-n}B) & \text{ if $n< 0$,}\\
                                  \KK^\ind_0(\Sigma^nA,B) & \text{ if $n\geq 0$.}
                                 \end{cases}
$$
If $A,B\in\Ho\Csep^{\op}$ there is a natural isomorphism
\begin{equation}\label{e:KK}
\KK^\ind_n(A,B)\cong\KK_n(B,A),
\end{equation}
where the right hand side denotes Kasparov's $\KK$-theory.

As noted after Definition \ref{d:homology theory}, if we pick any object $V$ in $\KKI$ we obtain a cohomology theory on $\Ind(\Csep^{\op})$ by $$\K^n_V=\KKI(\Sigma^{-n}\circ \LPK(-),V),$$
where $\Sigma$ here denotes the suspension functor in $\KKI$.
In particular, choosing $V=\LPK(\mathbb{C})$, we obtain a cohomology theory on $\Ind(\Csep^{\op})$ which we denote
$$\K^n=\KKI(\Sigma^{-n}\circ \LPK(-),\LPK(\mathbb{C}))=\begin{cases}
                                  \KKI(\Sigma^{-n}\circ \LPK(-),\LPK(\mathbb{C})) & \text{ if $n< 0$,}\\
                                  \KKI( \LPK(-),\Sigma^{n}\circ\LPK(\mathbb{C})) & \text{ if $n\geq 0$.}
                                 \end{cases}=$$
$$\begin{cases}
\KK^\ind_0(\Sigma^{-n}(-),\mathbb{C}) & \text{ if $n< 0$,}\\
\KK^\ind_0(-,\Sigma^n\mathbb{C}) & \text{ if $n\geq 0$.}
\end{cases}=
\KK^\ind_{-n}(-,\mathbb{C}).$$
For $n=0$ we obtain
$$\K^0=\KKI(\LPK(-),\LPK(\mathbb{C}))=\KK^\ind_{0}(-,\mathbb{C}).$$

There is also a corresponding cohomology theory on $\SpN(\Ind(\Csep^\op))$ given by
$$\K^0:=\KKI(\mathbb{L}\id(-),\LPK(\mathbb{C})).$$
(Note that since $\SpN(\Ind(\Csep^\op))$ is stable this determines $\K^n$ for all $n$.) As noted in Remark \ref{r:Brown}, by the general Brown-Neeman representability theorem, this cohomology theory is representable, that is, there exists an object $U$ in $\Ho\SpN(\Ind(\Csep^\op))$ and a natural isomorphism
$$\K^0\cong\Ho\SpN(\Ind(\Csep^\op))(-,U).$$
But as also noted there, we can actually give a rather explicit description of a representing object. We have isomorphisms, natural in $A\in\Ho\SpN(\Ind(\Csep^\op))$, using the functor $G_0(-)$ constructed at the end of Appendix \ref{s:stab_infinity}
$$\K^0(A)=\KKI(\mathbb{L}\id(A),\LPK(\mathbb{C}))\cong
\pi_0\Map_{\mathbf{KK}^\ind_{\infty}}(A^c,G_0(\mathbb{C}))\cong$$
$$\pi_0\Map_{\mathbf{KK}^\ind}(A^c,G_0(\mathbb{C})^f)\cong
\pi_0\Map_{\SpN(\Ind(\Csep^\op))}(A^c,G_0(\mathbb{C})^f)\cong$$
$$\pi_0\Map_{\SpN(\Ind(\Csep^\op))_\infty}(A,G_0(\mathbb{C})^f)\cong
\Ho\SpN(\Ind(\Csep^\op))(A,G_0(\mathbb{C})^f).$$
where $(-)^c$ denotes a functorial cofibrant replacement in $\SpN(\Ind(\Csep^\op))$ and $G_0(\mathbb{C})^f$ denotes a fibrant replacement of $G_0(\mathbb{C})$ in $\mathbf{KK}^\ind$. Thus we see that $G_0(\mathbb{C})^f$ is a representing object for the cohomology theory $\K^0$ on $\SpN(\Ind(\Csep^\op))$.

\begin{rem}
It is plausible that $\mathbf{KK}^\ind$ is a model for the stable $\infty$-category $\mathtt{KK_\infty}$ that was constructed by the third author in \cite{MahColoc}.
\end{rem}

\begin{rem}
One can use the $\KK^\ind$-theory to define a bivariant $\K$-theory for certain pro-$C^*$-algebras. Let $\cL: \TPro(\Csep)\map\Pro(\Csep)$ be the functor constructed in Section \ref{s:connect}. For two objects $A$ and $B$ in $\TPro(\Csep)$ we define, in analogy to Equation \ref{e:KK} above
$$\KK_*(B,A):= \KK_*^\ind(\cL(A),\cL(B)).$$
This $\KK$-theory will agree with the bivariant $\K$-theory for separable $\sC$-algebras \cite{Cun} that was denoted by $\skk$-theory in \cite{MahKK} (not to be confused with the diffotopy invariant bivariant $\K$-theory for locally convex algebras) on a reasonably large subcategory (cf. Theorem \ref{t:main Bonkat} below and Proposition 36 of \cite{MahKK}).
\end{rem}

\subsubsection{Other triangulated homology theories}
Repeating the procedure of the previous subsection with other sets $S$ of morphisms in $\Csep^\op$ we obtain other stable model categories, and induced triangulated homology theories on $\Ind(\Csep^\op)$, extending well known triangulated homology theories on $\Csep^\op$. We list a few examples.
\begin{enumerate}
\item If we define
$$S'_1 := \{\theta(\phi)^\op\;|\; 0\map A\map B\overset{\phi}{\map} C\map 0 \text{ is an extension in }\Csep\},$$
and take the set $S$ to be $S'_1\cup S_2$ we obtain an extension of Connes-Higson bivariant $\E$-theory category.
\item  If we take the set $S$ to be just $S'_1,$
we obtain an extension of the noncommutative stable homotopy category $\NSH$. The corresponding $\infty$-category is the stable $\infty$-category of noncommutative spectra constructed in \cite{MahNSH}.
\item Let $M_2(\mathbb{C})$ be the $C^*$ algebra of $2\times 2$ matrices over $\mathbb{C}$. For any $A\in\Csep$ there is an induced morphism $\chi_A:A\map A\otimes M_2(\mathbb{C})$ in $\Csep$, sending $a$ to $a\otimes e_{11}$. We define a set of morphisms in $\Csep^\op$ by
$$S'_2:= \{\chi_A^\op\;|\; A\in\Csep\}.$$
If we take the set $S$ to be $S_1\cup S'_2$ we obtain an extension of the \emph{connective} bivariant $\K$-theory category.
\item If we take the set $S$ to be $S'_1\cup S'_2$ we obtain an extension of the \emph{connective} bivariant $\E$-theory category that is modelling the (opposite of the) stable $\infty$-category $\mathtt{E^{cn}_\infty}$ of \cite[Section 3]{MahNSp}.

\end{enumerate}

\section{Comparison with Bonkat's bivariant $\K$-theory category}\label{s:Bonkat}
In the previous subsection we have constructed a bivariant $\K$-theory that is applicable to {\em all} projective systems of separable $C^*$-algebras.
In \cite{Bon}, Bonkat constructed a bivariant $\K$-theory that is applicable to projective systems of separable $C^*$-algebras that have surjective connecting homomorphisms and admit a countable cofinal subsystem.
In this subsection we will show that our $\K$-theory agrees with Bonkat's construction in certain cases, and admits better formal properties.
We first recall some facts about Bonkat's construction.

Let $\Pro_\Bon(\Csep)$ denote the full subcategory of $\Pro(\Csep)$ spanned by the objects $X:J\to \Csep$ that have surjective connecting homomorphisms and such that there exists a countable cofiltered category $K$ and a cofinal functor $K\to J$. In \cite{Bon} Bonkat constructed an additive category $\BKK$ and a pointed functor $\H:\Pro_\Bon(\Csep)\map\BKK$. Extending Higson's universal characterization of $\KK$-theory \cite{Hig} it is shown in \cite[Satz 3.5.10]{Bon} that the functor $\H:\Pro_\Bon(\Csep)\map\BKK$ is the universal additive category valued functor that has the following properties:

 \begin{enumerate}
  \item Homotopy invariance, i.e., the functor $\H$ is invariant under simplicial homotopy. Simplicial homotopy is the homotopy relation between maps generated by the standard path object given by the underlying simplicial structure. See Definition \ref{d:action}.
  \item $C^*$-stability, i.e., for any $\{A_j\}\in \Pro_\Bon(\Csep)$, and any minimal projection $p\in\cpt$, the induced morphism $\H(\{A_j\})\map \H(\{A_j\otimes\cpt\})$ is an isomorphism in $\BKK$.
  \item Split exactness, i.e., whenever \beqn \xymatrix{0\ar[r] & \{A_i\}\ar[r] & \{B_j\}\ar[r]^g & \{C_k\}\ar@/_1pc/[l]_s\ar[r] & 0}\eeqn is a split exact sequence in $\Pro_\Bon(\Csep)$, then $\H\{B_j\}\cong\H \{A_i\}\oplus \H\{C_k\}$ in $\BKK$.
 \end{enumerate}

\begin{rem}
In Bonkat's notation the category $\Pro_\Bon(\Csep)$ is denoted $\mathbf{S}\mathcal{C}_{\all}$, the category $\BKK$ is denoted $\mathbf{KK}_{\mathbf{S}\mathcal{C}_{\all}}$ and the functor $\H$ is denoted $KK_{\mathcal{C}_{\all}}$.
See the beginning of Section 2.3, Section 2.4 and Definition 3.5.1 in \cite{Bon}.
\end{rem}

We constructed above a triangulated homology theory on $\Ind(\Csep^\op)$,
$$\mathbb{L}\pi_\K:\Ho\Ind(\Csep^\op)\xrightarrow{}\KKI.$$
It will be more convenient for us now to work with the opposite functor
$$\LPK^{\op}:\Ho\Pro(\Csep)\rightarrow\KKP.$$
We denote the composition
$$\Pro(\Csep)\to\Ho\Pro(\Csep)\xrightarrow{\LPK^{\op}}\KKP$$
also by $\LPK^{\op}$.

We denote the restriction of $\LPK^{\op}$ to $\Pro_\Bon(\Csep)$ by
$$\tau:\Pro_\Bon(\Csep)\to\Ho\Pro(\Csep)\xrightarrow{\LPK^{\op}}\KKP.$$

We now wish to show that $\tau$ has homotopy invariance, $C^*$-stability and split exactness.
It will certainly be enough to show the following:
\begin{thm}\label{t:LPK satisfy}
The functor $\LPK^{\op}:\Pro(\Csep)\to \KKP$ has the following properties:
\begin{enumerate}
\item Invariance under simplicial homotopy.
\item For any $\{A_j\}\in \Pro(\Csep)$, and any minimal projection $p\in\cpt$, the induced morphism $\LPK^{\op}\{A_j\}\map \LPK^{\op}\{A_j\otimes\cpt\}$ is an isomorphism in $\KKP$.
\item Whenever \beqn \xymatrix{0\ar[r] & \{A_i\}\ar[r] & \{B_j\}\ar[r]^g & \{C_k\}\ar@/_1pc/[l]_s\ar[r] & 0}\eeqn is a split exact sequence in $\Pro(\Csep)$, then $\LPK^{\op}\{B_j\}\cong\LPK^{\op} \{A_i\}\oplus \LPK^{\op}\{C_k\}$ in $\KKP$.
 \end{enumerate}
\end{thm}

We will need the following lemma:
\begin{lem}\label{l:lim WE}
A cofiltered limit of weak equivalences, in the category of morphisms of $\mathbf{KK}^{\pro}$, is a weak equivalence.
\end{lem}
\begin{proof}
By \cite[Proposition 3.6]{RaRo} it is enough to show that $\mathbf{KK}^{\ind}$ has a generating set of cofibrations between finitely presentable objects. This is easily seen by following the construction of $\mathbf{KK}^{\ind}$.
\end{proof}

\begin{proof}[Proof of Theorem \ref{t:LPK satisfy}]\
\begin{enumerate}
\item Let $\{A_i\}$ and $\{B_j\}$ be objects in $\Pro(\Csep)$ and let $f,g:\{A_i\}\to\{B_j\}$ be simplicially homotopic maps in $\Pro(\Csep)$. We need to show that $\LPK^{\op}f=\LPK^{\op}g$. We have $f,g:\{B_j\}\to\{A_i\}$ as morphisms in $\Pro(\Csep)^{\op}\simeq\Ind(\Csep^{\op})$. There exists a morphism $H:\Delta^1\otimes\{B_j\}\to\{A_i\}$ in $\Ind(\Csep^{\op})$ such that $H\circ i_0=f$ and $H\circ i_1=g$.  We need to show that $\LPK f=\LPK g$. Clearly it is enough to show that $\chi f=\chi g$ in $\Ho\Ind(\Csep^{\op})$, where $\chi:\Ind(\Csep^{\op})\to \Ho\Ind(\Csep^{\op})$ is the natural functor.

    We now wish to show that $i_0:\{B_j\}\to\Delta^1\otimes\{B_j\}$ and $ p:\Delta^1\otimes\{B_j\}\to\{B_j\}$ are inverse simplicial homotopy equivalences in $\Ind(\Csep^{\op})$.
    Clearly $p\circ i_0=\id_{\{B_j\}}$ so it is enough to show that $ i_0\circ p$ is simplicially homotopic to $\id_{\Delta^1\otimes\{B_j\}}$.
 We define
 $$K:\Delta^1\otimes(\Delta^1\otimes \{B_j\})\cong(\Delta^1\times\Delta^1)\otimes \{B_j\}\to \Delta^1\otimes \{B_j\}$$
 to be the map that is induced by the simplicial map $\Delta^1\times \Delta^1\to \Delta^1$ that sends $(0,0),(0,1),(1,0)$ to $0$ and $(1,1)$ to $1$.
Clearly $K$ is a simplicial homotopy from $ i_0\circ p$ to $\id_{\Delta^1\otimes\{B_j\}}$.

Since $\Ind(\Csep^{\op})$ is a simplicial model category, we know that every simplicial homotopy equivalence is a weak equivalence (see for example \cite[Proposition 9.5.16]{Hir}). Thus we obtain that $i_0:\{B_j\}\to\Delta^1\otimes\{B_j\}$ and $ p:\Delta^1\otimes\{B_j\}\to\{B_j\}$ are weak equivalences. It follows that $\chi i_0:\{B_j\}\to\Delta^1\otimes\{B_j\}$ and $\chi p:\Delta^1\otimes\{B_j\}\to\{B_j\}$ are isomorphisms in $\Ho\Ind(\Csep^{\op})$. Since $\chi p\circ \chi i_0=\chi (p\circ i_0)=\chi \id=\id$, we know that they are inverse isomorphisms. By a similar argument we obtain that $\chi i_1:\{B_j\}\to\Delta^1\otimes\{B_j\}$ and $\chi p:\Delta^1\otimes\{B_j\}\to\{B_j\}$ are inverse isomorphisms in $\Ho\Ind(\Csep^{\op})$. Thus we obtain
$$ \chi i_0=(\chi p)^{-1}=\chi i_1$$
in $\Ho\Ind(\Csep^{\op})$.

We now see that we have
$$\chi f=\chi (H\circ i_0)=\chi H\circ \chi i_0=\chi H\circ \chi i_1=\chi (H\circ i_1)=\chi g$$
in $\Ho\Ind(\Csep^{\op})$.
\item Let $\{A_j\}\in \Pro(\Csep)$, and let $p\in\cpt$ be a minimal projection. By Lemma \ref{l:cofinal_CDS} there exists a small
cofinite directed set $A$ and a cofinal functor $A\to J$.
We pull back the morphism $\{A_j\}\map \{A_j\otimes\cpt\}$ along the cofinal functor $A\to J$ and obtain a  morphism in $\Csep^A$ which we denote by $\{B_a\}\map \{B_a\otimes\cpt\}$.
By employing the construction described in~\cite[Definition 4.3]{BarSch0} we have a functorial factorization of the morphisms in $\Csep^A$ into a map in $Lw(\cW)$ followed by a map in $Sp(\cF)$.
We apply this functorial factorization to the morphisms $\{B_a\}\to *$ and $\{B_a\otimes\cpt\}\to *$ in $\Csep^A$, and obtain the following diagrams in $\Csep^A$:
$$\{B_a\}\xrightarrow{Lw(\cW)}\{B_a^f\}\xrightarrow{Sp(\cF)} *,$$
$$\{B_a\otimes\cpt\}\xrightarrow{Lw(\cW)}\{(B_a\otimes\cpt)^f\}\xrightarrow{Sp(\cF)} *.$$
Note that $\{B_a^f\}$ and $\{(B_a\otimes\cpt)^f\}$ are fibrant, as objects in the model category $\Pro(\Csep)$ (see Theorem \ref{t:proC_main}).
By the functoriality of the factorization we obtain a commutative square in $\Csep^A$ of the form
$$\xymatrix{\{B_a\}\ar[d]^{Lw(\cW)}\ar[r] & \{B_a\otimes\cpt\}\ar[d]^{Lw(\cW)}\\
\{B_a^f\}\ar[r] & \{(B_a\otimes\cpt)^f\},}$$
where the upper horizontal map is induced by the minimal projection $p$. This is also a square in $\Pro(\Csep)^A$, so we can
 apply the right Quillen functor
$$\pi_\K^{\op}:\Pro(\Csep)\cong\Ind(\Csep^{\op})^{\op}\xrightarrow{G_0^{\op}}\SpN(\Ind(\Csep^{\op}))^{\op}
\xrightarrow{\id}\mathbf{KK}^{\pro}$$
objectwise on this square and then take the limit in $\mathbf{KK}^{\pro}$. We obtain a diagram in $\mathbf{KK}^{\pro}$ of the form
$$\xymatrix{{\lim}_{{a\in A}}^{{\mathbf{KK}^{\pro}}}\pi_\K^{\op}(B_a)\ar[d]\ar[r] & {\lim}_{{a\in A}}^{{\mathbf{KK}^{\pro}}}\pi_\K^{\op}(B_a\otimes\cpt)\ar[d]\\
{\lim}_{{a\in A}}^{{\mathbf{KK}^{\pro}}}\pi_\K^{\op}(B_a^f)\ar[r] & {\lim}_{{a\in A}}^{{\mathbf{KK}^{\pro}}}\pi_\K^{\op}((B_a\otimes\cpt)^f).}$$

It follows from Proposition \ref{p:homology_main}, that for every $a\in A$ the map $\pi_\K^{\op}(B_a)\to \pi_\K^{\op}(B_a\otimes\cpt)$ is a weak equivalence in $\mathbf{KK}^{\pro}$. (Actually, Proposition \ref{p:homology_main} only shows this for the map induced by the specific minimal projection by which we localized; but it follows from, for instance, Lemma 2.1 of \cite{HigETh} that any two minimal projections will produce homotopic maps.) By Lemma \ref{l:lim WE}, we get that
$${\lim}_{{a\in A}}^{{\mathbf{KK}^{\pro}}}\pi_\K^{\op}(B_a)\to {\lim}_{{a\in A}}^{{\mathbf{KK}^{\pro}}}\pi_\K^{\op}(B_a\otimes\cpt)$$
is also a weak equivalence in $\mathbf{KK}^{\pro}$.

Being a right Quillen functor, $\pi_\K^{\op}$ transfers weak equivalences between fibrant objects to weak equivalences. Since every object in $\Csep$ is fibrant in $\Pro(\Csep)$, we see that for every $a\in A$ the map $\pi_\K^{\op}(B_a)\to \pi_\K^{\op}(B_a^f)$ is a weak equivalence in $\mathbf{KK}^{\pro}$. By Lemma \ref{l:lim WE}, we get that
$${\lim}_{{a\in A}}^{{\mathbf{KK}^{\pro}}}\pi_\K^{\op}(B_a)\to {\lim}_{{a\in A}}^{{\mathbf{KK}^{\pro}}}\pi_\K^{\op}(B_a^f)$$
is also a weak equivalence in $\mathbf{KK}^{\pro}$.
By the same argument one shows that
$${\lim}_{{a\in A}}^{{\mathbf{KK}^{\pro}}}\pi_\K^{\op}(B_a\otimes\cpt)\to {\lim}_{{a\in A}}^{{\mathbf{KK}^{\pro}}}\pi_\K^{\op}((B_a\otimes\cpt)^f)$$
is a weak equivalence in $\mathbf{KK}^{\pro}$.

From the two out of three property in $\mathbf{KK}^{\pro}$ and the fact that $\pi_\K^{\op}$ commutes with limits, we get that
$$\pi_\K^{\op}(\{B_a^f\})\cong \pi_\K^{\op}({\lim}_{{a\in A}}^{\Pro(\Csep)}B_a^f)\cong{\lim}_{{a\in A}}^{{\mathbf{KK}^{\pro}}}\pi_\K^{\op}(B_a^f)$$
$$\map {\lim}_{{a\in A}}^{{\mathbf{KK}^{\pro}}}\pi_\K^{\op}((B_a\otimes\cpt)^f)\cong \pi_\K^{\op}({\lim}_{{a\in A}}^{\Pro(\Csep)}(B_a\otimes\cpt)^f)
\cong \pi_\K^{\op}(\{(B_a\otimes\cpt)^f\})$$
is a weak equivalence in $\mathbf{KK}^{\pro}$.

Since $\{B_a^f\}$ is a fibrant replacement to $\{A_j\}$ and $\{(B_a\otimes\cpt)^f\}$ is a fibrant replacement to $\{A_j\otimes\cpt\}$, in the model category $\Pro(\Csep)$, the map $(\LPK)^{\op}(\{A_j\})\map (\LPK)^{\op}(\{A_j\otimes\cpt\})$ is isomorphic to
$\pi_\K^{\op}(\{B_a^f\})\to \pi_\K^{\op}(\{(B_a\otimes\cpt)^f\})$ as a morphism in $\KKP$, so we get (2).

\item Let
\beqn \xymatrix{0\ar[r] & \{A_i\}\ar[r] & \{B_j\}\ar[r]^g & \{C_k\}\ar@/_1pc/[l]_s\ar[r] & 0}\eeqn
be a split exact sequence in $\Pro(\Csep)$. We need to show that $(\LPK)^{\op}\{B_j\}\cong(\LPK)^{\op} \{A_i\}\oplus (\LPK)^{\op}\{C_k\}$ in $\KKP$.
It is enough to show that
\beqn \xymatrix{(\LPK)^{\op}\{A_i\}\ar[r] & (\LPK)^{\op}\{B_j\}\ar[rr]^{(\LPK)^{\op} g} & & (\LPK)^{\op}\{C_k\}\ar@/_2pc/[ll]_{(\LPK)^{\op} s}} \eeqn is part of a triangle in $\KKP$ (because then, this triangle clearly splits so $(\LPK)^{\op}\{B_j\}\cong(\LPK)^{\op} \{A_i\}\oplus (\LPK)^{\op}\{C_k\}$).

We denote by $\cT$ the category freely generated by the following graph
\beqn \xymatrix{0\ar[r]^a & 1\ar@/_1pc/[l]_b}\eeqn
with the single relation that $a\circ b=\id_1$. This is a finite category (finite number of morphisms). By \cite[Section 4]{Mey}, we have that the natural functor $$\Pro(\Csep^\cT)\to\Pro(\Csep)^\cT$$
is an equivalence of categories. The following diagram
\beqn \xymatrix{\{B_j\}\ar[r]^g & \{C_k\}\ar@/_1pc/[l]_s}\eeqn
gives an object in $\Pro(\Csep)^\cT$. Thus, from the equivalence of categories above, we get that there exists a cofitered category $L$ and a diagram in $\Csep^L$ of the form
\beqn \xymatrix{ \{Y_l\}\ar[r]^f & \{Z_l\}\ar@/_1pc/[l]_t }\eeqn
such that for any $l\in L$ we have  $f_l\circ t_l=id_{Z_l}$, that is isomorphic to
\beqn \xymatrix{\{B_j\}\ar[r]^g & \{C_k\}\ar@/_1pc/[l]_s}\eeqn
as an object in $\Pro(\Csep)^\cT$.

Let $\{X_l\}$ denote the levelwise kernel of $f$
\beqn \xymatrix{\{X_l\}\ar[r] & \{Y_l\}\ar[r]^f & \{Z_l\}. }\eeqn
Since this is also the kernel of $f$ in $\Pro(\Csep)$, we obtain a commutative diagram in $\Pro(\Csep)$
\beqn \xymatrix{\{X_l\}\ar[r]\ar[d]^\cong & \{Y_l\}\ar[r]^f\ar[d]^\cong & \{Z_l\}\ar@/_1pc/[l]_t\ar[d]^\cong\\
                \{A_i\}\ar[r] & \{B_j\}\ar[r]^g & \{C_k\}\ar@/_1pc/[l]_s}\eeqn
such that the vertical maps are isomorphisms.

By Lemma \ref{l:cofinal_CDS} there exists a small
cofinite directed set $A$ and a cofinal functor $A\to L$.
We pull back the diagram $\{X_l\}\map \{Y_l\}\map \{Z_l\}$ along the cofinal functor $A\to L$ and obtain a diagram in $\Csep^A$ which we denote by $\{X_a\}\map \{Y_a\}\map \{Z_a\}$.

We now follow a line of arguments similar to the one used in (2) above, where it is explained in more detail. We begin by employing the functorial factorization in $\Csep^A$, into a map in $Lw(\cW)$ followed by a map in $Sp(\cF)$, and obtain a commutative diagram in $\Csep^A$ of the form
$$\xymatrix{\{X_a\}\ar[d]^{Lw(\cW)}\ar[r] & \{Y_a\}\ar[d]^{Lw(\cW)} \ar[r] & \{Z_a\}\ar[d]^{Lw(\cW)}\\
\{X_a^f\}\ar[r] & \{Y_a^f\}\ar[r] & \{Z_a^f\}}$$
such that $\{X_a^f\}$, $\{Y_a^f\}$ and $\{Z_a^f\}$ are fibrant, as objects in $\Pro(\Csep)$. Applying $\pi_\K^{\op}$ objectwise and taking the limit in $\mathbf{KK}^{\pro}$ we obtain a diagram in $\mathbf{KK}^{\pro}$
$$\xymatrix{{\lim}_{{a\in A}}^{{\mathbf{KK}^{\pro}}}\pi_\K^{\op}(X_a)\ar[d]\ar[r] & {\lim}_{{a\in A}}^{{\mathbf{KK}^{\pro}}}\pi_\K^{\op}(Y_a)\ar[d]\ar[r] & {\lim}_{{a\in A}}^{{\mathbf{KK}^{\pro}}}\pi_\K^{\op}(Z_a)\ar[d]\\
{\lim}_{{a\in A}}^{{\mathbf{KK}^{\pro}}}\pi_\K^{\op}(X_a^f)\ar[r] & {\lim}_{{a\in A}}^{{\mathbf{KK}^{\pro}}}\pi_\K^{\op}(Y_a^f)\ar[r] & {\lim}_{{a\in A}}^{{\mathbf{KK}^{\pro}}}\pi_\K^{\op}(Z_a^f).}$$

For every $a\in A$ the map $\pi_\K^{\op}(X_a)\to \pi_\K^{\op}(X_a^f)$ is a weak equivalence in $\mathbf{KK}^{\pro}$. By Lemma \ref{l:lim WE}, it follows that
$${\lim}_{{a\in A}}^{{\mathbf{KK}^{\pro}}}\pi_\K^{\op}(X_a)\to {\lim}_{{a\in A}}^{{\mathbf{KK}^{\pro}}}\pi_\K^{\op}(X_a^f)$$
is also a weak equivalence in $\mathbf{KK}^{\pro}$.
By the same argument one shows that
$${\lim}_{{a\in A}}^{{\mathbf{KK}^{\pro}}}\pi_\K^{\op}(Y_a)\to {\lim}_{{a\in A}}^{{\mathbf{KK}^{\pro}}}\pi_\K^{\op}(Y_a^f),\:\:{\lim}_{{a\in A}}^{{\mathbf{KK}^{\pro}}}\pi_\K^{\op}(Z_a)\to {\lim}_{{a\in A}}^{{\mathbf{KK}^{\pro}}}\pi_\K^{\op}(Z_a^f)$$
are weak equivalence in $\mathbf{KK}^{\pro}$.

From the fact that $\pi_\K^{\op}$ commutes with limits we get that the diagram $\pi_\K^{\op}\{X_a\}\to \pi_\K^{\op}\{Y_a\}\to \pi_\K^{\op}\{Z_a\}$ is isomorphic to $\pi_\K^{\op}\{X_a^f\}\to \pi_\K^{\op}\{Y_a^f\}\to \pi_\K^{\op}\{Z_a^f\}$ in $\KKP$.

Since $\{X_a^f\}$, $\{Y_a^f\}$ and $\{Z_a^f\}$ are fibrant replacements for $\{A_i\}$, $\{B_j\}$ and $\{C_k\}$, in the model category $\Pro(\Csep)$, the diagram
$$(\LPK)^{\op}\{A_i\}\map (\LPK)^{\op}\{B_j\}\map (\LPK)^{\op}\{C_k\}$$
is isomorphic to
$$\pi_\K^{\op}\{X_a^f\}\to \pi_\K^{\op}\{Y_a^f\}\to \pi_\K^{\op}\{Z_a^f\}$$
as a diagram in $\KKP$. So we are left to show that
$$\pi_\K^{\op}\{X_a\}\to \pi_\K^{\op}\{Y_a\}\to \pi_\K^{\op}\{Z_a\}$$
is part of a triangle in $\KKP$.

We apply the above functorial factorization to the morphism $\{Y_a\}\to \{Z_a\}$ in $\Csep^A$, and obtain
$\{Y_a\}\xrightarrow{Lw(\cW)} \{Y'_a\}\xrightarrow{Sp(\cF)} \{Z_a\}$.
By \cite[Proposition 2.19]{BarSch1} we know that the morphism $\{Y'_a\}\xrightarrow{} \{Z_a\}$ is levelwise in $\cF$. Thus, for every $a\in A$ we obtain a factorization $Y_a\xrightarrow{\cW} Y'_a\xrightarrow{\cF} Z_a$, in $\Csep$. Let $\{X'_a\}$ denote the levelwise fiber of $\{Y'_a\}\xrightarrow{} \{Z_a\}$.

Let $a\in A$.
\beqn \xymatrix{X_a\ar[r]& Y_a\ar[r]^{f_a} & Z_a\ar@/_1pc/[l]_{t_a}}\eeqn
is a split exact sequence in $\Csep$. In particular, it is a cpc-split exact sequence, so the map
$$\pi_\K^{\op}(X_a)\to \pi_\K^{\op}(X'_a)$$
is a weak equivalence in $\mathbf{KK}^{\pro}$ (see Proposition \ref{p:homology_main}). By Lemma \ref{l:lim WE} and the fact that $\pi_\K^{\op}$ commutes with limits, we get that
$$\pi_\K^{\op}\{X_a\}\cong {\lim}_{{a\in A}}^{{\mathbf{KK}^{\pro}}}\pi_\K^{\op}(X_a)\to {\lim}_{{a\in A}}^{{\mathbf{KK}^{\pro}}}\pi_\K^{\op}(X'_a)\cong \pi_\K^{\op}\{X'_a\}$$
is also a weak equivalence in $\mathbf{KK}^{\pro}$.

For every $a\in A$ the map $Y_a\to Y'_a$ is a weak equivalence between fibrant objects in $\Pro(\Csep)$, so the map
$$\pi_\K^{\op}(Y_a)\to \pi_\K^{\op}(Y'_a)$$
is a weak equivalence in $\mathbf{KK}^{\pro}$. By Lemma \ref{l:lim WE} and the fact that $\pi_\K^{\op}$ commutes with limits, we get that
$$\pi_\K^{\op}\{Y_a\}\cong {\lim}_{{a\in A}}^{{\mathbf{KK}^{\pro}}}\pi_\K^{\op}(Y_a)\to {\lim}_{{a\in A}}^{{\mathbf{KK}^{\pro}}}\pi_\K^{\op}(Y'_a)\cong \pi_\K^{\op}\{Y'_a\}$$
is also a weak equivalence in $\mathbf{KK}^{\pro}$.

We thus obtain the following diagram in $\mathbf{KK}^{\pro}$:
\beqn \xymatrix{\pi_\K^{\op}\{X_a\}\ar[r]\ar[d]_\sim & \pi_\K^{\op}\{Y_a\}\ar[d]_{\sim}\ar[r] & \pi_\K^{\op}\{Z_a\}\ar[d]_=\\
\pi_\K^{\op}\{X'_a\}\ar[r] & \pi_\K^{\op}\{Y'_a\}\ar[r] & \pi_\K^{\op}\{Z_a\}.}\eeqn
It follows that it is enough to show that
$$\pi_\K^{\op}\{X'_a\}\to \pi_\K^{\op}\{Y'_a\}\to \pi_\K^{\op}\{Z_a\}$$
is part of a triangle in $\KKP$. But this follows from the fact that $\{Y'_a\}\to \{Z_a\}$ is a fibration in $\Pro(\Csep)$ and $\pi_\K^{\op}$ is a right Quillen functor.
\end{enumerate}
\end{proof}

We have thus shown that $\tau:\Pro_\Bon(\Csep)\to\KKP$ has homotopy invariance, $C^*$-stability and split exactness.
It follows that there exists a unique additive functor $i:\BKK\to \KKP$ such that the following diagram commutes
$$\xymatrix{\Pro_\Bon(\Csep)\ar[rr]^\H \ar[drr]_{\tau} & & \BKK\ar[d]^i\\
              & & \KKP.}$$

We will now bring two computational tools for calculating Bonkat's $\K$-theory for diagrams. The first is a Milnor type $\lim^1$-sequence.
\begin{thm}[{\cite[Satz 4.5.4]{Bon}}]\label{t:Bonkat lim1}
Let $\{A_n\}_{n\in\NN}$ be a sequence of nuclear separable $C^*$-algebras with surjective connecting $*$-homomorphisms $A_{n+1}\to A_n$, and let $\{B_j\}$ be an object of $\Pro_\Bon(\Csep)$. Then there is a natural short exact sequence
$$0\to {\lim}^1_n \BKK(\H\{B_j\}, \H\Sigma A_n)\to \BKK(\H\{B_j\},\H\{A_n\})\to {\lim}_n \BKK(\H\{B_j\}, \H A_n)\to 0.$$
\end{thm}

\begin{thm}[{\cite[Satz. 4.5.5]{Bon}}]\label{t:Bonkat lim}
Let $\{A_n\}_{n\in\NN}$ be a sequence of nuclear separable $C^*$-algebras with surjective connecting $*$-homomorphisms $A_{n+1}\to A_n$, and let $B$ be an object in $\Csep$. Then there exists a natural isomorphism
$$\colim_n \BKK(\H A_n, \H B)\cong \BKK(\H\{ A_n\}, \H B).$$
\end{thm}

We will now show that our $\K$-theory, namely $(\LPK)^{\op}:\Pro(\Csep)\to \KKP$, also has the same type of computational tools as Bonkat's, but in an even more general setting.
\begin{thm}\label{t:Our lim1}
Let $\{A_n\}_{n\in\NN}$ be a sequence of separable $C^*$-algebras (that need not be nuclear) with connecting $*$-homomorphisms $A_{n+1}\to A_n$ (that need not be surjective), and let $\{B_j\}$ be an object of $\Pro(\Csep)$. Then there is a natural short exact sequence
$$0\to {\lim}^1_n \KKP(\LPK^{\op}\{B_j\}, \LPK^{\op}\Sigma A_n)\to \KKP(\LPK^{\op}\{B_j\},\LPK^{\op}\{A_n\})\to {\lim}_n \KKP(\LPK^{\op}\{B_j\}, \LPK^{\op}A_n)\to 0.$$
\end{thm}

\begin{proof}
We need to show that
there is a natural short exact sequence
$$0\to {\lim}^1_n \KKI(\LPK\Sigma A_n, \LPK\{B_j\})\to \KKI(\LPK\{A_n\},\LPK\{B_j\})\to {\lim}_n \KKI(\LPK A_n,\LPK\{B_j\} )\to 0.$$

Let $\mathbb{N}$ denote the cofinite directed poset of natural numbers.
By employing the construction described in~\cite[Definition 4.3]{BarSch0} we have a functorial factorization of the morphisms in $\Csep^\NN$ into a map in $Lw(\cW)$ followed by a map in $Sp(\cF)$.
We apply this functorial factorization to the morphisms $\{A_n\}\to *$ in $\Csep^\NN$, and obtain the following diagram in $\Csep^\NN$:
$$\{A_n\}\xrightarrow{Lw(\cW)}\{A^f_n\}\xrightarrow{Sp(\cF)}*.$$
By \cite[Proposition 2.17]{BarSch1}, we know that every map $A^f_{n+1}\to A^f_n$ is a Schochet fibration. Thus, we have a sequence of cofibrations
$$*\xrightarrow{} A^f_0\xrightarrow{} A^f_1 \xrightarrow{}\cdots\xrightarrow{} A^f_n\xrightarrow{}\cdots$$
in the pointed model category $\Ind(\Csep^{\op})$, with colimit $\{A_n^f\}$. It follows that
$$*\xrightarrow{} \pi_\K A^f_0\xrightarrow{} \pi_\K A^f_1 \xrightarrow{}\cdots\xrightarrow{} \pi_\K A^f_n\xrightarrow{}\cdots$$
is a sequence of cofibrations in the pointed model category $\mathbf{KK}^{\ind}$, with colimit $\pi_\K\{A_n^f\}$.
By \cite[Proposition 7.3.2]{HovBook}, for every fibrant $Y\in\mathbf{KK}^{\ind}$ we have an exact sequence
$$0\xrightarrow{}{\lim}^1_n[\Sigma (\pi_\K A_n^f) ,Y ] \xrightarrow{} [\pi_\K\{A_n^f\},Y ] \xrightarrow{}{\lim}_n[\pi_\K A_n^f ,Y ] \xrightarrow{}0.$$
(Note that for every $X\in \Ind(\Csep^{\op})$ we have $S^1\wedge (\pi_\K X)\cong \pi_\K(S^1\wedge X)$. See Theorem \ref{t:spectra}.)
\end{proof}

\begin{thm}\label{t:Our lim}
Let $\{A_j\}_{n\in J}$ be an object in $\Pro(\Csep)$ and let $B$ be an object in $\Csep$. Then there exists a natural isomorphism
$$\colim_j \KKP(\LPK^{\op} A_j, \LPK^{\op} B)\cong \KKP(\LPK^{\op} \{ A_j\}, \LPK^{\op} B).$$
\end{thm}

\begin{proof}
For every $X\in \cS$ and every $A\in \Set$ we have a natural isomorphism
$$\Hom_\Set(\pi_0(X),A)\simeq\Map_{\cS_{\infty}}(X,D(A)),$$
where $D(A)$ denotes the constant simplicial set on $A$.
Thus, there is an adjunction between $\infty$-categories
$$\pi_0:\cS_{\infty}\rightleftarrows\Ne(\Set):D.$$
It follows that $\pi_0:\cS_{\infty}\to\Ne(\Set)$ commutes with $\infty$-colimits. Thus we have natural isomorphisms
$$\KKP(\LPK^{\op} \{ A_j\}, \LPK^{\op} B)\simeq\pi_0\Map_{\mathbf{KK}^{\pro}_{\infty}}(\LPK^{\op} \{ A_j\}, \LPK^{\op} B)\simeq\pi_0\Map_{\mathbf{KK}^{\pro}_{\infty}}(\LPK^{\op} {\lim}_j^{\infty}A_j, \LPK^{\op} B)\simeq$$
$$\simeq\pi_0\Map_{\mathbf{KK}^{\pro}_{\infty}}({\lim}_j^{\infty}\LPK^{\op} A_j, \LPK^{\op}B)\simeq\pi_0{\colim}_j^{\infty} \Map_{\mathbf{KK}^{\pro}_{\infty}}(\LPK^{\op} A_j, \LPK^{\op} B)\simeq$$
$$\simeq{\colim}_j^{\infty}\pi_0\Map_{\mathbf{KK}^{\pro}_{\infty}}(\LPK^{\op} A_j, \LPK^{\op} B)\simeq{\colim}_j\KKP(\LPK^{\op} A_j, \LPK^{\op} B).$$

In the diagram above we take the derived functors in the higher categorical sense:
$$\mathbb{L}\pi_K:\Ind(\Csep^{\op})_{\infty}\rightleftarrows
\mathbf{KK}^{\ind}_{\infty}:\mathbb{R}\chi_K$$
where $\pi_K:=\id\circ G_0$ and $\chi_K:=Ev_0\circ\id$.
The fact that
$$\{ A_j\}\cong {\lim}_j A_j\simeq {\lim}_j^{\infty}A_j$$
in $\Ind(\Csep^{\op})_{\infty}$ follows from the fact that the model category $\Ind(\Csep^{\op})$ has a generating set of cofibrations between finitely presentable objects. The fact that $\LPK^{\op}B$ is compact in $\mathbf{KK}^{\ind}_{\infty}$ follows from the fact that $B$ is compact in $\Ind(\Csep^{\op})_{\infty}$ and $\mathbb{R}\chi_K$ commutes with filtered colimits.
\end{proof}

\begin{rem}
Theorems \ref{t:Our lim1} and \ref{t:Our lim} remain true for all triangulated homology theories defined in Section \ref{ss:THT}.
\end{rem}

We are now ready to state our result connecting Bonkat's $\K$-theory and ours.
\begin{thm}\label{t:main Bonkat}
Let $\{A_n\}_{n\in\NN}$ and $\{B_m\}_{m\in\NN}$ be sequences of nuclear separable $C^*$-algebras with surjective connecting $*$-homomorphisms $A_{n+1}\to A_n$ and $B_{m+1}\to B_m$. Then $i:\BKK\to\KKP$ induces a natural isomorphism
$$\BKK(\H\{B_m\},\H\{A_n\})\cong \KKP(\LPK^{\op}\{B_m\},\LPK^{\op}\{A_n\}).$$
\end{thm}

\begin{proof}
By Theorems \ref{t:Bonkat lim1} and \ref{t:Our lim1} we get a commutative diagram
$$\xymatrix{
 \lim^1_n \BKK(\H\{ B_m\}, \H\Sigma A_n)\ar@{^{(}->}[r]\ar[d]& \BKK(\H\{ B_m\},\H\{A_n\} )\ar@{->>}[r]\ar[d]& \lim_n \BKK(\H\{ B_m\}, \H A_n)\ar[d]\\
 \lim^1_n \KKP(\LPK^{\op}\{ B_m\},  \LPK^{\op} \Sigma A_n)\ar@{^{(}->}[r] &\KKP(\LPK^{\op}\{ B_m\},\LPK^{\op} \{ A_n \})\ar@{->>}[r] & \lim_n \KKP(\LPK^{\op}\{ B_m\}, \LPK^{\op} A_n).
 }$$

By the Five Lemma it suffices to show that the extremal vertical arrows above are isomorphisms.

By Theorem \ref{t:Bonkat lim} there is a natural isomorphism
$$\colim_m \BKK(\H B_m, \H A_n)\cong \BKK(\H\{ B_m\}, \H A_n).$$

By Theorem \ref{t:Our lim} there is a natural isomorphism
$$\colim_m \KKP(\LPK^{\op} B_m, \LPK^{\op} A_n)\cong \KKP(\LPK^{\op} \{ B_m\}, \LPK^{\op} A_n).$$

Since $\KKP$ and $\BKK$ both agree with Kasparov $\KK$-theory for separable $C^*$-algebras, we conclude that the right vertical arrow in the diagram above is an isomorphism. A similar argument shows that the left vertical arrow is also an isomorphism and hence we are done.
\end{proof}

\begin{rem}
Using Theorem \ref{t:main Bonkat} and the results of \cite[Kapitel 5]{Bon}, it is possible to compare our $\K$-theory with other extensions of Kasparov's $\K$-theory considered in the literature.
\end{rem}

\appendix

\section{Model categories}\label{a:model}

In this appendix we recall the notion of model categories and some of their theory that we need in this paper. For the basic theory the reader is referred to \cite{HovBook}, \cite{Hir} and the appendix of \cite{Lur}.

\begin{defn}
A \emph{model category} is a quadruple $(\cM,\cW,\cF,\cC)$ satisfying the following:
\begin{enumerate}
\item $\cM$ is a complete and cocomplete category.
\item $\cW,\cF,\cC$ are subcategories of $\cM$ that are closed under retracts.
\item $\cW$ satisfies the two out of three property.
\item $\cC\cap \cW\subseteq{}^{\perp}\cF$  and $\cC\subseteq{}^{\perp}(\cF\cap\cW)$.
\item There exist functorial factorizations of the morphisms in $\cM$ into a map in $\cC\cap \cW$ followed by a map in $\cF$, and into a map in $\cC$ followed by a map in $\cF\cap \cW$.
\end{enumerate}
\end{defn}

\begin{defn}\label{d:combinatorial}
Let $(\cM,\cW,\cF,\cC)$ be a model category. Then the model category $\cM$ is called \emph{combinatorial} if it is locally presentable (see \cite{AR}) and there are sets $I$ and $J$ of morphisms in $\cM$ (called generating cofibrations and generating acyclic cofibrations) such that $\cF=J^{\perp}$ and $\cF\cap\cW=I^{\perp}$. In particular, a combinatorial model category is cofibrantly generated (see \cite[Definition 2.1.17]{HovBook}).
\end{defn}

\subsection{Simplicial model categories}

\begin{defn}\label{d:hom_map}
Let $\cM$ and $\cC$ be categories. An \emph{adjunction of two variables} from $\cM\times \cC$ to $\cC$ is a quintuple $(\otimes,\Map,\hom,\phi_r,\phi_l)$, where
$$(-)\otimes (-):\cM\times \cC\to \cC,$$
$$\Map(-,-):\cC^{\op}\times\cC\to \cM,$$
$$\hom(-,-):\cM^{\op}\times \cC\to \cC$$
are bifunctors, and $\phi_r,\phi_l$ are natural isomorphisms
$$\phi_r: \cC(K \otimes X,Y)\xrightarrow{\cong} \cM(K,\Map(X,Y)),$$
$$\phi_l: \cC(K \otimes X,Y)\xrightarrow{\cong} \cC(X,\hom(K,Y)).$$
In the sequel we will suppress the natural isomorphisms $\phi_r,\phi_l$ and write the adjunction of two variables just as $(\otimes,\Map,\hom).$
\end{defn}

\begin{defn}\label{d:LQB}
Let $\cM$ and $\cC$ be model categories and let $(-)\otimes (-):\cM\times \cC\to \cC$ be a bifunctor.
The bifunctor $\otimes$ is called a \emph{left Quillen bifunctor} if $\otimes$ is a part of a two variable adjunction $(\otimes,\Map,\hom)$, and for every cofibration $j:K\to L$ in $\cM$ and every cofibration $i:X\to Y$ in $\cC$ the induced map
$$K \otimes Y \coprod_{K \otimes X}L\otimes X\to L\otimes Y$$
is a cofibration (in $\cC$), which is acyclic if either $i$ or $j$ is.
\end{defn}

\begin{prop}[{\cite[Lemma 4.2.2]{HovBook}}]\label{p:Qbifunc}
Let $\cM$ and $\cC$ be model categories. Let $(\otimes,\Map,\hom)$ be a two variable adjunction.
Then the following conditions are equivalent:
\begin{enumerate}
\item The bifunctor $\otimes$ is a left Quillen bifunctor.
\item For every cofibration $j:K\to L$ in $\cM$ and every fibration $p:A\to B$ in $\cC$, the induced map:
$$\hom(L,A)\to \hom(K,A) \prod_{\hom(K,B)}\hom(L,B)$$
is a fibration (in $\cC$), which is acyclic if either $j$ or $p$ is.
\item For every cofibration $i:X\to Y$ in $\cC$ and every fibration $p:A\to B$ in $\cC$ the induced map:
$$\Map(Y,A)\to \Map(X,A) \prod_{\Map(X,B)}\Map(Y,B)$$
is a fibration (in $\cM$), which is acyclic if either $i$ or $p$ is.
\end{enumerate}
\end{prop}

\begin{defn}\label{d:simplicial sets}
Let $\cS=\Set^{\Delta^{\op}}$ denote the category of simplicial sets. The category $\cS$ has a standard model structure where a map $X\to Y$ in $\cS$ is:
\begin{enumerate}
\item A cofibration, if it is one to one (at every degree).
\item A weak equivalence, if the induced map of geometric realizations $|X|\to |Y|$ is a weak equivalence of topological spaces.
\item A fibration, if it has the right lifting property with respect to all acyclic cofibrations.
\end{enumerate}
\end{defn}

\begin{defn}\label{d:simplicial}
A \emph{simplicial model category} is a model category $\cC$ together with a left Quillen bifunctor $\otimes:\cS\times \cC\to \cC$ and coherent natural isomorphisms
$$L\otimes(K\otimes X)\cong (K\times L)\otimes X,$$
$$\Delta^0\otimes X\cong X,$$
for $X$ in $\cC$ and $K,L$ in $\cS$.
\end{defn}

\subsection{Left and right proper model categories}
\begin{defn}[{\cite[Section A.2.4]{Lur}}]
A model category $\cC$ is called:
\begin{enumerate}
\item \emph{Left proper}, if for every push out square in $\cC$ of the form
\[
\xymatrix{A\ar[d]^i\ar[r]^f & B\ar[d]^j\\
C\ar[r] & D,}
\]
such that $i$ is a weak equivalence and $f$ is a cofibration, the map $j$ is also a weak equivalence.
\item \emph{Right proper}, if for every pull back square in $\cC$ of the form
\[
\xymatrix{C\ar[d]^j\ar[r] & D\ar[d]^i\\
A\ar[r]^f & B,}
\]
such that $i$ is a weak equivalence and $f$ is a fibration, the map $j$ is also a weak equivalence.
\item Proper, if it is both left and right proper.
\end{enumerate}
\end{defn}

\subsection{Pointed simplicial model categories}\label{s:cofiber}

Recall that a category is called \emph{pointed} if it has a zero object, that is, an object which is both initial and terminal.

Let $\cM$ be any pointed simplicial model category. It follows from the general theory of simplicial model categories that $\cM$ can be turned naturally into an $\cS_*$-enriched model category, where $\cS_*=(\cS_*,\wedge,S^0)$ is the symmetric monoidal model category of \emph{pointed} simplicial sets. (This just means that we replace $\cS$ by $\cS_*$ and $\times$ by $\wedge$ in Definition \ref{d:simplicial}.)
Thus, for every $A$ and $B$ in $\cM$ there is a pointed simplicial set $\Map_*(A,B)$. Actually we have:
$$\Map_*(A,B)=\Map(A,B),$$
as simplicial sets, where the distinguished morphism from $A$ to $B$ is the zero morphism, given by the composition:
$$A\to 0\to B.$$
Moreover, for every $A,B,C$ in $\cM$ the pointed enriched composition
$$\circ:\Map_*(B,C)\wedge \Map_*(A,B)\to \Map_*(A,C),$$
is just the quotient of the unpointed composition $\circ:\Map(B,C)\times \Map(A,B)\to \Map(A,C).$

Furthermore, for every object $A$ in $\cM$ and every pointed simplicial set $K$ we have the pointed left and right actions:
$$K\wedge A\in \cM\:\:,\:\:\hom_*(K,A)\in \cM.$$
It can be shown that for every (unpointed) simplicial set $K$ we have natural isomorphisms
$$K_+\wedge A\cong K\otimes A\:\:,\:\: \hom_*(K_+,A)\cong\hom(K,A),$$
where $K_+$ denotes $K$ with a disjoint basepoint.

The cofiber of a map in $\cM$ is defined to be the coequalizer of this map with the zero map. In the pointed simplicial model category $\cS_*$ we define the object $S^1$ as
$$S^1:=\cofib(\partial\Delta^1_+\hookrightarrow \Delta^1_+)\in \cS_*$$
Since $\cM$ is an $\cS_*$-enriched model category and $S^1$ is cofibrant in $\cS_*$, we have a Quillen pair
$$S^1\wedge(-):\cM\rightleftarrows\cM:\hom_*(S^1,-).$$
We define $\Sigma$ and $\Omega$ to be the adjoint pair of derived functors induced by this Quillen pair
$$\Sigma:=\mathbb{L}(S^1\wedge(-)):\Ho\cM\rightleftarrows
\Ho\cM:\mathbb{R}(\hom_*(S^1,-))=:\Omega.$$
Thus, for every object $A$ in $\cM$ we have:
$$\Sigma A\cong S^1\wedge A^c\:\:,\:\: \Omega A\cong\hom_*(S^1,A^f),$$
where $A^c$ and $A^f$ are any cofibrant and fibrant replacements for $A$ respectively.

\subsection{Left Bousfield localizations of model categories}\label{s:Bousfield_model}

Let $\cM$ be a simplicial model category. It follows that $\Ho\cM$ is naturally enriched tensored and cotensored over the monoidal category $(\Ho\cS,\times,*)$.
\begin{defn}\label{d:bousfield}
Let $T$ be a class of morphisms in $\Ho\cM$.
 \begin{enumerate}
 \item An object $W$ in $\Ho\cM$ is called $T$-\emph{local} if for every element $f:A\to B$ in $T$ the induced map $$f^*:\mathbb{R}\Map(B,W)\to \mathbb{R}\Map(A,W)$$
     is an isomorphism in $\Ho\cS$.
 \item A morphism $g:X\to Y$ in $\Ho\cM$ is called a $T$-\emph{local isomorphism} if for every $T$-local object $W$ in $\Ho\cM$ the induced map
     $$g^*:\mathbb{R}\Map(Y,W)\to \mathbb{R}\Map(X,W)$$
     is an isomorphism in $\Ho\cS$.
 \item A morphism $g:X\to Y$ in $\cM$ is called a $T$-\emph{local equivalence} if the induced morphism $X\to Y$ in $\Ho\cM$ is a $T$-local isomorphism.
 \item If the cofibrations in $\cM$ and the $T$-local equivalences constitute a model structure on $\cM$ then the \emph{left Bousfield localization} of $\cM$ with respect to $T$ is said to exist and is defined to be this model structure and denoted $\mathrm{L}_T \cM$.
 \end{enumerate}
\end{defn}

\begin{rem}
Sometimes we will apply Definition \ref{d:bousfield} and other results on Bousfield localization to a class of morphisms $T$ in $\cM$, the intended meaning being that we are considering the image of $T$ under the natural functor $\cM\to\Ho\cM$.
\end{rem}

The following proposition is shown in \cite{Hir} Propositions 3.3.5, 3.3.16, 3.4.1, 3.4.4 and Theorem 3.3.19.
\begin{prop}\label{p:bousfield}
Let $T$ be a class of morphisms in $\Ho\cM$ and suppose that the  left Bousfield localization of $\cM$ with respect to $T$ exists. Then the following hold:
    \begin{enumerate}
    \item If $\cM$ is left proper then $\mathrm{L}_T \cM$ is also left proper and the fibrant objects in $\mathrm{L}_T \cM$ are precisely the fibrant objects in $\cM$ that are $T$-local as objects in $\Ho\cM$.
    \item The left Quillen functor $\id:\cM\to \mathrm{L}_T \cM$ is initial among left Quillen functors $F:\cM\to \cN$ such that $\mathbb{L}F$ transfers morphisms in $T$ to isomorphisms in $\Ho\cN$. That is, if $F:\cM\to \cN$ is a left Quillen functor as above, then $F$ itself is also a left Quillen functor from $\mathrm{L}_T \cM$ to $\cN$.
    \end{enumerate}
\end{prop}

We now state the main theorem in the theory of left Bousfield localizations. It is shown in \cite[Proposition A.3.7.3]{Lur} (see also \cite[Theorem 4.1.1]{Hir}).
\begin{thm}\label{t:bousfield}
Suppose that $\cM$ is left proper and combinatorial. Then the left Bousfield localization of $\cM$ with respect to any small set $T$  of morphisms in $\Ho\cM$ exists and is again combinatorial. Moreover the model category $\mathrm{L}_T \cM$ is simplicial, with the same simplicial structure as $\cM$.
\end{thm}

\subsection{Stabilization of model categories}\label{s:stab_model}

In this subsection we recall the notion of a stable model category and the process of stabilization in the world of model categories. We will be using results from \cite{Hov}.

A pointed simplicial model category $\cM$ is called stable if the suspension functor $\Sigma:\Ho\cM\to \Ho\cM$ is an equivalence of categories, or in other words, if the Quillen pair
$$S^1\wedge(-):\cM\rightleftarrows\cM:\hom_*(S^1,-),$$
is a Quillen equivalence.

Let $\cM$ be any pointed simplicial model category. It is desirable to have at our disposal a stable model category that is as close to $\cM$ as possible. This can be achieved using a construction of Hovey \cite{Hov}, provided $\cM$ satisfies the following conditions:
\begin{enumerate}
\item $\cM$ is left proper.
\item $\cM$ is combinatorial.
\item The domains of the generating cofibrations of $\cM$ can be taken to be cofibrant.
\end{enumerate}
(The results in \cite{Hov} are stated under the assumption that $\cM$ is {\em cellular} but according to the results in \cite[Section A.3.7]{Lur}, it suffices that it is combinatorial.) In the notation of \cite{Hov} the category that we need is $\SpN(\cM,S^1)$, but we denote it here simply by $\SpN(\cM)$. We sketch the construction of $\SpN(\cM)$ and the natural functor $G_0:\cM\to \SpN(\cM)$.

An object of $\SpN(\cM)$ is a sequence $\{X_0,X_1,\dots\}$ of objects of $\cM$ together with structure maps $S^1\wedge X_n\map X_{n+1}$.
A morphism $\{X_0,X_1,\dots\}\to \{Y_0,Y_1,\dots\}$ in $\SpN(\cM)$ consists of a sequence of morphisms $X_n\to Y_n$ preserving the structure maps.

We now define a model structure on $\SpN(\cM)$ which is called the stable model structure. We begin with the projective model structure on $\SpN(\cM)$ in which a morphism $\{X_0,X_1,\dots\}\to \{Y_0,Y_1,\dots\}$ is a weak equivalence or fibration if $X_n\to Y_n$ is a weak equivalence or fibration for every $n$.

An object $\{X_0,X_1,\dots\}$ of $\SpN(\cM)$ is called an $\Omega$-spectrum if for every $n$ the map $X_n\map \hom_*(S^1,X_{n+1})$, adjoint to the structure map $S^1\wedge X_n\map X_{n+1}$, is a weak equivalence.

The stable structure on $\SpN(\cM)$ is obtained from the projective structure by a process of left Bousfield localization (see Definition \ref{d:bousfield}). We take the left Bousfield localization in such a way that the fibrant objects in the localized model structure are precisely the projective fibrant objects that are also $\Omega$-spectra.

For every $n\geq 0$ we have a Quillen adjunction
$$G_n:\cM\rightleftarrows \SpN(\cM):Ev_n,$$
where $Ev_n$ is the evaluation functor sending the object $\{X_0,X_1,\dots\}$ to $X_n$, and $G_n$ is its left adjoint.
The functor $G_0$
sends $X$ to the sequence of objects $\{X,S^1\wedge X,\cdots,  S^n\wedge X,\dots \}.$

The following Proposition follows from \cite[Theorem 6.3 and the paragraph before it]{Hov}, \cite[Corollary 6.5]{Hov} and \cite[Theorem 10.3]{Hov}.
\begin{thm}\label{t:spectra}
The model category $\SpN(\cM)$ is stable, left proper, simplicial and combinatorial.
The functors $G_n:\cM\map\SpN(\cM)$ are left Quillen and preserve the simplicial action up to a natural isomorphism.
\end{thm}

\section{$\infty$-categories}\label{a:infinity}

In this appendix we recall the notion of ${\infty}$-categories and some of their theory that we need in this paper. Our approach is based on quasi-categories, and  the reader is referred to \cite{Lur} for the basic theory.

\begin{defn}[Joyal, Lurie]
An \emph{$\infty$-category} is a simplicial set $\cC$ satisfying the right lifting property with respect to the maps $\Lambda^n_i \to \Delta^n$ for $0 < i < n$ (where $\Lambda^n_i$ is the simplicial set obtained by removing from $\partial\Delta^n$ the $i$'th face). If $\cC$ and $\cD$ are $\infty$-categories, then an $\infty$-functor $\cC\to\cD$ is just a simplicial set map. In fact, we have an $\infty$-category of $\infty$-functors from $\cC$ to $\cD$ denoted $\Fun(\cC,\cD)$ and defined by
$$\Fun(\cC,\cD)_n:=\Hom_\cS(\Delta^n\times\cC,\cD).$$
\end{defn}

\subsection{Relative categories and their associated ${\infty}$-categories}\label{ss:infinity}

In this subsection we will recall the notion of \emph{${\infty}$-localization} which associates an underlying ${\infty}$-category to any relative category.  The material here is based on \cite{Hin}.

\begin{defn}\label{d:rel}
A \emph{relative category} is a category $\cC$ equipped with a subcategory
\[\cW \subseteq \cC\]
containing all the identities. We will refer to the maps in $\cW$ as \emph{weak equivalences}.
\end{defn}

Given a relative category $(\cC,\cW)$ one may associate to it an ${\infty}$-category $\cC_{\infty} = \cC[\cW^{-1}]$, equipped with a map $\cC \xrightarrow{} \cC_{\infty}$, which is characterized by the following universal property: for every ${\infty}$-category $\cD$, the natural map
$$ \Fun(\cC_{\infty},\cD) \xrightarrow{} \Fun(\cC,\cD) $$
is fully-faithful, and its essential image is spanned by those functors $\cC \xrightarrow{} \cD$ which send $\cW$ to equivalences. The ${\infty}$-category $\cC_{\infty}$ is called the \emph{${\infty}$-localization} of $\cC$ with respect to $\cW$. In this paper we will also refer to $\cC_{\infty}$ as the \emph{underlying ${\infty}$-category of $\cC$}, or the ${\infty}$-category \emph{modelled by $\cC$}. We note that this notation and terminology is slightly abusive, as it makes no direct reference to $\cW$.

The ${\infty}$-category $\cC_{\infty}$ may be constructed in one of the following equivalent ways:
\begin{enumerate}
\item
One may construct the \emph{Hammock localization} of $\cC$ with respect to $\cW$ (see~\cite{DK}), and obtain a simplicial category $L^\H(\cC,\cW)$. The ${\infty}$-category $\cC_{\infty}$ can then be obtained by taking the coherent nerve of any fibrant model of $L^\H(\cC,\cW)$ (with respect to the Bergner model structure).
\item
One may consider the \emph{marked simplicial set} $\Ne_+(\cC,\cW) = (\Ne(\cC),\cW)$, where $\Ne$ denotes the nerve functor. The ${\infty}$-category $\cC_{\infty}$ can then be obtained by taking the underlying simplicial set of any fibrant model of $\Ne_+(\cC,\cW)$ (with respect to the Cartesian model structure, see~\cite[Chapter 3]{Lur}).
\end{enumerate}

\subsection{Stabilization of $\infty$-categories}\label{s:stab_infinity}
In this subsection we consider the notion of stabilization of $\infty$-categories.
The following is based on the very accessible presentation of Harpaz \cite{Har}. For a more detailed account see \cite{Lur2}.

Let $\Catfinc$ denote the (big) $\infty$-category of pointed finitely cocomplete small $\infty$-categories and finite-colimit-preserving functors between them. If $\cC$ is an object in $\Catfinc$ then we can define the \emph{suspension functor} on $\cC$
$$\Sigma_\cC:\cC\to \cC$$
by the formula
$$\Sigma_\cC(X):=*\coprod_X *.$$

We define $\Catex$ to be the full subcategory of
$\Catfinc$ spanned by the objects where the suspension functor is an equivalence.
$\Catex$ is called the $\infty$-category of small stable $\infty$-categories and exact functors between them.

Let $\cC$ be an object in $\Catfinc$. We denote by $\SW(\cC)$ the colimit of the sequence
$$\cC\xrightarrow{\Sigma_\cC}\cC\xrightarrow{\Sigma_\cC}\cdots$$
in the $\infty$-category $\Catfinc$. In fact, $\SW(\cC)$ is also the colimit of the sequence above in $\Catinfty$, which is the $\infty$-category of all small $\infty$-categories and all $\infty$-functors between them. Thus, the objects
of $\SW(\cC)$ are pairs $(X,n)$ where $X\in\cC$ and $n\in \mathbb{N}$, and the mapping spaces
are given by
$$\Map_{\SW(\cC)}((X,n),(Y,m)) = \colim_k\Map_\cC(\Sigma_\cC^{k-n} X,\Sigma_\cC^{k-m} Y),$$
where the colimit is taken in the $\infty$-category of spaces.
This construction will yield a left adjoint to the inclusion $\Catex\to\Catfinc$.
More precisely, we have a unit map
$$\Sigma^{\infty}_\cC:\cC\to \SW(\cC)$$
given by $X\mapsto (X,0)$, which satisfies the following universal property: For every stable $\infty$-category $\cD$, pre-composition with $\Sigma^{\infty}_\cC$
induces an equivalence of $\infty$-categories
$$\Funex(\SW(\cC),\cD)\to\Funfinc(\cC,\cD).$$

Let $\Catfinl$ denote the (big) $\infty$-category of pointed finitely complete small $\infty$-categories and finite-limit-preserving functors between them. If $\cC$ is an object in $\Catfinl$ then we can define the \emph{loop functor} on $\cC$
$$\Omega_\cC:\cC\to \cC$$
by the formula
$$\Omega_\cC(X):=*\prod_X *.$$

It can be shown that the $\infty$-category $\Catex$ is equivalent to the full subcategory of
$\Catfinl$ spanned by the objects where the loop functor is an equivalence.

We will denote by $\Sp(\cC)$ the limit of the tower
$$\cC\xleftarrow{\Omega_\cC}\cC\xleftarrow{\Omega_\cC}\cdots$$
in the $\infty$-category $\Catfinl$. In fact, $\Sp(\cC)$ is also the limit in $\Catinfty$, namely, an object
of $\Sp(\cC)$ is given by a sequence $\{X_n \}$ of objects of $\cC$ together with equivalences
$X_n\simeq\Omega_\cC X_{n+1}$ and maps are given by compatible families of maps.

This construction will yield a right adjoint to the inclusion $\Catex\to\Catfinl$.
More precisely, we have a counit map
$$\Omega^{\infty}_\cC:\Sp(\cC)\to \cC$$
given by  $\{X_n \}\mapsto X_0$, which satisfies the following universal property: For every stable $\infty$-category $\cD$, composition with $\Omega^{\infty}_\cC$
induces an equivalence of $\infty$-categories
$$\Funex(\cD,\Sp(\cC))\to\Funfinl(\cD,\cC).$$

We now discuss the process of stabilization in the context of presentable $\infty$-categories. Let $\Prl$ denote the (big) $\infty$-category of pointed presentable $\infty$-categories and left functors between them (i.e. functors which admit right adjoints) and $\Prr$ the $\infty$-category of pointed presentable $\infty$-categories and right functors between
them (i.e. functors which admit left adjoints). The categories $\Prl$
and $\Prr$ are naturally opposite to each other.
The adjoint functor theorem for presentable $\infty$-categories tells us that a
functor $f:\cC\to\cD$ between presentable $\infty$-categories is a left functor if and
only if it preserves all colimits and is a right functor if and only if it is accessible
and preserves all limits. In particular, if $\cC$ and $\cD$ are stable presentable $\infty$-categories
then any left functor between them and any right functor between them is
exact. We will denote by $\Prlex\subseteq \Prl$ the full subcategory spanned by the stable $\infty$-categories and similarly by $\Prrex\subseteq \Prr$.

Observe that for a pointed presentable $\infty$-category $\cC$ the following are equivalent:
\begin{enumerate}
\item $\cC$ is stable.
\item $\Sigma_\cC$ is an equivalence.
\item $\Omega_\cC$ is an equivalence.
\end{enumerate}
We thus see that in order to perform the stabilization process inside the world of pointed presentable $\infty$-categories one just needs to invert
either the suspension or the loop functor. As above, this can be done from the
left or from the right. However, since $\Prl$ and $\Prr$
are opposite to each other,
it will be enough to understand just one of these procedures. In this case the
right option has an advantage, and that is that limits in $\Prr$
can be computed
just as limits in $\Catinfty$ (where the same is not true for colimits in $\Prl$).

Now the functor $\Omega_\cC$ has a left adjoint $\Sigma_\cC$, so we see that $\Omega_\cC$ is a right functor,
i.e., a legitimate morphism in $\Prr$. As above, we can invert it by taking the
inverse limit of the tower
$$\cC\xleftarrow{\Omega_\cC}\cC\xleftarrow{\Omega_\cC}\cdots$$
in the $\infty$-category $\Prr$. Fortunately, this procedure is the same as computing
the limit in $\Catinfty$, i.e., it will coincide with $\Sp(\cC)$ described above. However,
we are now guaranteed that $\Sp(\cC)$ will be a presentable $\infty$-category and that
the projection map
$$\Omega^{\infty}_\cC:\Sp(\cC)\to \cC$$
will be a right functor of presentable $\infty$-categories. Now if $\cD$ is any stable
presentable $\infty$-category then composition with $\Omega^{\infty}_\cC$
induces an equivalence of
$\infty$-categories
$$\Funr(\cD,\Sp(\cC))\to \Funr (\cD,\cC)$$
The duality between $\Prr$ and $\Prl$ means that we can automatically get a dual result with no extra work. Namely, the left adjoint
$$\Sigma^{\infty}_\cC:\cC\to\Sp(\cC)$$
of $\Omega^{\infty}_\cC$
will also exhibit $\Sp(\cC)$ as a stabilization of $\cC$ from the left in the $\infty$-category $\Prl$. In other words, if $\cD$ is any stable presentable $\infty$-category then
pre-composition with $\Sigma^{\infty}_\cC$ induces an equivalence of $\infty$-categories
$$\Funl(\Sp(\cC),\cD)\to \Funl(\cC,\cD).$$

Now suppose that $\cC\in\Prl$ is also compactly generated, i.e. it
is of the form $\Ind(\cC_0 )$ where $\cC_0$ is a small pointed ${\infty}$-category with finite colimits. Then one can attempt to left-stabilize $\cC$ by first left-stabilizing $\cC_0$ using the construction $\SW(\cC_0)$ considered previously, and then considering its ind-completion $\Ind(\SW(\cC_0))$. This construction will yield again a stable presentable $\infty$-category satisfying the same universal property as $\Sp(\cC)$. We will hence deduce that there is a natural equivalence
$$\Ind(\SW(\cC_0 ))\simeq \Sp(\cC).$$
Note that in the equivalence above we are referring to the $\infty$-categorical construction of the ind-category (see \cite[Section 5.3]{Lur}).

We now wish to connect the $\infty$-categorical stabilization presented above to the model categorical stabilization presented in Appendix \ref{s:stab_model}.

Let $\cM$ be a left proper combinatorial pointed simplicial model category such that the domains of the generating cofibrations of $\cM$ can be taken to be cofibrant. Since $\cM$ is combinatorial, it follows from \cite{Lur} that $\cM_\infty$ is a presentable $\infty$-category.

As explained in Appendix \ref{s:cofiber} we have a Quillen pair
$$S^1\wedge(-):\cM\rightleftarrows\cM:\hom_*(S^1,-).$$
By \cite[Theorem 2.1]{MG} this Quillen pair induces an adjoint pair of $\infty$-categories
$$\Sigma:=\mathbb{L}(S^1\wedge(-)):\cM_{\infty}\rightleftarrows
\cM_{\infty}:\mathbb{R}(\hom_*(S^1,-))=:\Omega.$$
(The adjoint pair of usual categories
$$\Sigma:\Ho\cM\rightleftarrows\Ho\cM:\Omega$$
considered in Appendix \ref{s:cofiber}, is obtained from the $\infty$-categorical adjoint pair by passing to the homotopy categories.)
Then we have $\Sigma=\Sigma_{\cM_{\infty}}$ and $\Omega=\Omega_{\cM_{\infty}}$. In particular, we see that $\cM$ is a stable model category iff $\cM_{\infty}$ is a stable $\infty$-category.

According to \cite[Proposition 4.15]{Rob}, we have a natural equivalence of $\infty$-categories
$$\SpN(\cM)_\infty\simeq \Sp(\cM_\infty).$$
Moreover, if we consider the Quillen adjunction
$$G_0:\cM\rightleftarrows \SpN(\cM):Ev_0,$$
then the adjunction between the underlying $\infty$-categories given by \cite[Theorem 2.1]{MG}
$$\Sigma^{\infty}:=\mathbb{L}G_0:\cM_\infty\rightleftarrows
\SpN(\cM)_\infty:\mathbb{R}Ev_0=:\Omega^{\infty},$$
is equivalent to the adjunction $(\Sigma^{\infty}_{\cM_\infty},\Omega^{\infty}_{\cM_\infty})$ defined above, under this natural equivalence.


\bibliographystyle{abbrv}
\bibliography{ourBib}

\end{document}